\documentclass[a4paper]{amsart}
\usepackage{graphicx}
\usepackage{amsmath}
\usepackage{amssymb}
\usepackage{oldgerm}
\usepackage{mathdots}
\usepackage{stmaryrd}
\usepackage{bm}
\usepackage[all]{xy}
\usepackage{color}
\usepackage{enumerate}
\usepackage{fourier}
\usepackage{hyperref}  
\usepackage{mathrsfs}
\usepackage{tikz}
\usetikzlibrary{arrows,%
  decorations.markings,%
  matrix,%
  shapes}

\newcommand{\Hom}{\operatorname{Hom}\nolimits}
\renewcommand{\Im}{\operatorname{Im}\nolimits}

\newcommand{\stmod}{\operatorname{\underline{mod}}\nolimits}
\newcommand{\Mod}{\operatorname{Mod}\nolimits}
\newcommand{\Ker}{\operatorname{Ker}\nolimits}

\newcommand{\Ho}{\operatorname{H}\nolimits}

\newcommand{\cx}{\operatorname{cx}\nolimits}

\newcommand{\proj}{\operatorname{proj}\nolimits}

\newcommand{\Fact}{\operatorname{\bf{Fact}}\nolimits}
\newcommand{\HFact}{\operatorname{\bf{HFact}}\nolimits}
\newcommand{\HSeq}{\operatorname{\bf{HSeq}}\nolimits}
\newcommand{\K}{\operatorname{\bf{K}}\nolimits}
\newcommand{\add}{\operatorname{add}\nolimits}

\newcommand{\Ktac}{\operatorname{\bf{K_{tac}}}\nolimits}

\newcommand{\A}{\mathscr{A}}
\newcommand{\C}{\mathscr{C}}
\newcommand{\D}{\mathscr{D}}

\newcommand{\Proj}{\mathscr{P}}
\newcommand{\Free}{\mathscr{F}}

\newcommand{\diagram}[3]{\matrix (#1) [matrix of math nodes,row
  sep={#2},column sep={#3},text height=1.5ex,text
  depth=0.25ex]}

\newtheorem{theorem}{Theorem}[section]

\newtheorem{corollary}[theorem]{Corollary}

\newtheorem{lemma}[theorem]{Lemma}

\theoremstyle{definition}

\theoremstyle{definition}

\newtheorem{example}[theorem]{Example}
\theoremstyle{definition}

\newtheorem{remark}[theorem]{Remark}
\theoremstyle{remark}

\theoremstyle{definition}

\theoremstyle{definition}

\theoremstyle{definition}

\begin{document}

\title{Categorical matrix factorizations}

\author{Petter Andreas Bergh and David A.\ Jorgensen}

\address{Petter Andreas Bergh \\ Institutt for matematiske fag \\
NTNU \\ N-7491 Trondheim \\ Norway} \email{petter.bergh@ntnu.no}
\address{David A.\ Jorgensen \\ Department of mathematics \\ University
of Texas at Arlington \\ Arlington \\ TX 76019 \\ USA}
\email{djorgens@uta.edu}



\thanks{Part of this work was done while we were visiting the Mittag-Leffler Institute in February and March 2015, only seven years ago. We would like to thank the organizers of the Representation Theory program for inviting us to spend time at that wonderful place.}

\begin{abstract}
In this paper we give a purely categorical construction of $d$-fold matrix factorizations of a natural transformation, for any even integer $d$. This recovers the classical definition of those for regular elements in commutative rings due to Eisenbud. We explore some natural functors between associated triangulated categories, and show that when $d=2$ these are full and faithful, and in some cases equivalences.
\end{abstract}

\maketitle

\section{Introduction}
\label{sec:intro}

In 1980, Eisenbud introduced matrix factorizations in \cite{Eisenbud} for elements in commutative rings. The motivation was to study free resolutions over the corresponding factor rings, in particular hypersurface rings. It was shown that over such a ring, every minimal free resolution corresponds to a matrix factorization over the ambient regular ring. The homotopy category of matrix factorizations is triangulated, and it was remarked by Buchweitz in \cite{Buchweitz} that Eisenbud’s result implies that this homotopy category is equivalent to the singularity category of the corresponding hypersurface ring. In \cite{Orlov}, Orlov gave an explicit proof of this fact.

In this paper, for any even integer $d \ge 2$ we give a categorical construction of $d$-fold matrix factorizations for any suspended category and natural transformation commuting with suspension. Our construction recovers the classical notion of matrix factorization due to Eisenbud referenced above. It also recovers other notions in recent literature, in particular the $2$-fold factorizations in abelian categories defined in \cite{BallardEtAl}.
 
The paper is organized as follows. In Section \ref{sec:main} we give our construction of $d$-fold matrix factorizations, and show that the collection of such objects naturally forms an algebraic triangulated category. We also show that specific restrictions yield the previously studied triangulated categories of matrix factorizations.
 
In Section \ref{sec:tac} we study natural triangle functors between our triangulated category of matrix factorizations and other well-known triangulated categories. We show that when $d=2$, these functors are full and faithful, and in some contexts, equivalences. We end with some examples, which illustrate new sources of concrete matrix factorizations.

\section{Factorizations}
\label{sec:main}

Fix an even integer $d \ge 2$, and let $( \C, S )$ be a suspended additive category, that is, an additive category $\C$ together with an additive automorphism $S \colon \C \to \C$. Furthermore, fix a natural transformation $\eta \colon 1_{\C} \to S$ with the property that
$$\eta_{S(M)} = S ( \eta_M )$$
for every object $M$ of $\C$ (we shall usually drop parentheses when we apply $S$ to objects and morphisms).

A $d$-fold \emph{$( \C, S )$-factorization} $(M_i,f_i)$ of $\eta$ is a sequence
\begin{center}
\begin{tikzpicture}
\diagram{d}{3em}{4em}{
M_1 & M_2 & \cdots & M_d & SM_1 \\
 };
\path[->, font = \scriptsize, auto]
(d-1-1) edge node{$f_1$} (d-1-2)
(d-1-2) edge node{$f_2$} (d-1-3)
(d-1-3) edge node{$f_{d-1}$} (d-1-4)
(d-1-4) edge node{$f_d$} (d-1-5);
\end{tikzpicture}
\end{center}
of $d$ morphisms in $\C$, in which every $d$-fold composition equals $\eta$:
\begin{eqnarray*}
f_d \circ f_{d-1} \circ \cdots \circ f_1 & = & \eta_{M_1} \\
Sf_1 \circ f_d \circ \cdots \circ f_2 & = & \eta_{M_2} \\
& \vdots & \\
S f_{d-1} \circ \cdots \circ Sf_1 \circ f_d & = & \eta_{M_d}
\end{eqnarray*}
A \emph{morphism} $\varphi \colon (M_i,f_i) \to (N_i,g_i)$, between two such $(\C,S)$-factorizations of $\eta$, is a sequence $(\varphi_1, \dots, \varphi_d)$ of morphisms in $\C$ such that the diagram
\begin{center}
\begin{tikzpicture}
\diagram{d}{3em}{4em}{
M_1 & M_2 & \cdots & M_d & SM_1 \\
N_1 & N_2 & \cdots & N_d & SN_1 \\
 };
\path[->, font = \scriptsize, auto]
% upper horizontal maps
(d-1-1) edge node{$f_1$} (d-1-2)
(d-1-2) edge node{$f_2$} (d-1-3)
(d-1-3) edge node{$f_{d-1}$} (d-1-4)
(d-1-4) edge node{$f_d$} (d-1-5)
% lower horizontal maps
(d-2-1) edge node{$g_1$} (d-2-2)
(d-2-2) edge node{$g_2$} (d-2-3)
(d-2-3) edge node{$g_{d-1}$} (d-2-4)
(d-2-4) edge node{$g_d$} (d-2-5)
% vertical maps
(d-1-1) edge node{$\varphi_1$} (d-2-1)
(d-1-2) edge node{$\varphi_2$} (d-2-2)
(d-1-4) edge node{$\varphi_d$} (d-2-4)
(d-1-5) edge node{$S \varphi_1$} (d-2-5);
\end{tikzpicture}
\end{center}
commutes. It is an \emph{isomorphism} if all the $\varphi_i$, and hence also $S \varphi_1 $, are isomorphisms. The composition of two morphisms is a new morphism, thus we obtain the category $\Fact_d ( \C ,S, \eta )$ of $d$-fold $( \C ,S)$-factorizations of $\eta$. Since the original category $\C$ and its automorphism $S$ are additive, so is the category $\Fact_d ( \C ,S, \eta )$, in a natural way. Addition of morphisms is induced from $\C$, and the zero object is the trivial $( \C ,S)$-factorization $(0,0)$ of $\eta$. The coproduct of two $( \C ,S)$-factorizations of $\eta$ is performed pointwise, with the result being a new $( \C ,S)$-factorization since
$$\eta_{M \oplus N} = \left [ 
\begin{smallmatrix} \eta_M & 0 \\ 0 & \eta_N \end{smallmatrix} 
\right ]$$
for all objects $M,N \in \C$.

Let $\varphi = ( \varphi_i )$ and $\varphi' = ( \varphi_i' )$ be two parallell morphisms in $\Fact_d ( \C ,S, \eta )$, from $(M_i ,f_i)$ to $(N_i ,g_i)$. Then $\varphi$ is \emph{homotopic} to $\varphi'$ if there exist diagonal morphisms
\begin{center}
\begin{tikzpicture}
\diagram{d}{4em}{4.5em}{
M_1 & M_2 & \cdots & M_{d-1} & M_d & SM_1 \\
N_1 & N_2 & \cdots & N_{d-1} & N_d & SN_1 \\
 };
\path[->, font = \scriptsize, auto]
% upper horizontal maps
(d-1-1) edge node{$f_1$} (d-1-2)
(d-1-2) edge node{$f_2$} (d-1-3)
(d-1-3) edge node{$f_{d-2}$} (d-1-4)
(d-1-4) edge node{$f_{d-1}$} (d-1-5)
(d-1-5) edge node{$f_d$} (d-1-6)
% lower horizontal maps
(d-2-1) edge node{$g_1$} (d-2-2)
(d-2-2) edge node{$g_2$} (d-2-3)
(d-2-3) edge node{$g_{d-2}$} (d-2-4)
(d-2-4) edge node{$g_{d-1}$} (d-2-5)
(d-2-5) edge node{$g_d$} (d-2-6)
% vertical maps
(d-1-1) edge[bend right=20] node[left]{$\varphi_1$} (d-2-1)
(d-1-1) edge[bend left=20] node[right]{$\varphi_1'$} (d-2-1)
(d-1-2) edge[bend right=20] node[left]{$\varphi_2$} (d-2-2)
(d-1-2) edge[bend left=20] node[right]{$\varphi_2'$} (d-2-2)
(d-1-4) edge[bend right=20] node[left]{$\varphi_{d-1}$} (d-2-4)
(d-1-4) edge[bend left=20] node[right]{$\varphi_{d-1}'$} (d-2-4)
(d-1-5) edge[bend right=20] node[left]{$\varphi_d$} (d-2-5)
(d-1-5) edge[bend left=20] node[right]{$\varphi_d'$} (d-2-5)
(d-1-6) edge[bend right=20] node[left]{$S \varphi_1 $} (d-2-6)
(d-1-6) edge[bend left=20] node[right]{$S \varphi_1' $} (d-2-6)
% diagonal maps
(d-1-2) edge node[near start, left, xshift=-1mm]{$s_1$} (d-2-1)
(d-1-5) edge node[near start, left, xshift=-1mm]{$s_{d-1}$} (d-2-4)
(d-1-6) edge node[near start, left, xshift=-1mm]{$s_d$} (d-2-5);
\end{tikzpicture}
\end{center}
satisfying
\begin{eqnarray*}
\varphi_1 - \varphi_1' & = & s_1 \circ f_1 + S^{-1}g_d \circ S^{-1}s_d \\
\varphi_2 - \varphi_2' & = & s_2 \circ f_2 + g_1 \circ s_1 \\
& \vdots & \\
\varphi_d - \varphi_d' & = & s_d \circ f_d + g_{d-1} \circ s_{d-1} \\
\end{eqnarray*}
We write $\varphi \sim \varphi'$; this is an equivalence relation on the set of morphisms between the objects $(M_i ,f_i)$ and $(N_i, g_i)$. The equivalence class of $\varphi$ is denoted by $[ \varphi ]$. Note that if $\varphi, \varphi', \theta, \theta'$ are four morphisms in $\Fact_d ( \C,S,\eta )$ from $(M_i ,f_i)$ to $(N_i, g_i)$, with $\varphi \sim \varphi'$ and $\theta \sim \theta'$, then $( \varphi + \theta ) \sim ( \varphi' + \theta' )$. Consequently, homotopies are compatible with addition of morphisms, and addition $[ \varphi ] + [ \theta ] = [ \varphi + \theta ]$ of equivalence classes is therefore well defined. Moreover, homotopies are also compatible with composition of morphisms in $\Fact_d ( \C,S,\eta )$: if
\begin{center}
\begin{tikzpicture}
\diagram{d}{3em}{4em}{
(M_i,f_i) & (N_i,g_i) & (L_i,h_i) \\
 };
\path[->, font = \scriptsize, auto]
(d-1-1) edge[bend left] node{$\varphi$} (d-1-2)
(d-1-1) edge[bend right] node{$\varphi'$} (d-1-2)
(d-1-2) edge[bend left] node{$\psi$} (d-1-3)
(d-1-2) edge[bend right] node{$\psi'$} (d-1-3);
\end{tikzpicture}
\end{center}
are four morphisms with $\varphi \sim \varphi'$ and $\psi \sim \psi'$, then $( \psi \circ \varphi ) \sim ( \psi' \circ \varphi' )$. Thus composition $[ \psi ] \circ [ \varphi ] = [ \psi \circ \varphi ]$ of equivalence classes is well defined, too. Now form the homotopy category $\HFact_d ( \C,S,\eta )$ of $\Fact_d ( \C,S,\eta )$. Its objects are the same as the objects of $\Fact_d ( \C,S,\eta )$, but the morphisms are the homotopy equivalence classes of morphisms from $\Fact_d ( \C,S,\eta )$. We have just seen that this is indeed a category, whose morphism sets are abelian groups. The zero object (which is unique only up to homotopy) and coproducts are inherited from $\Fact_d ( \C,S,\eta )$, giving $\HFact_d ( \C,S,\eta )$ the natural structure of an additive category.

\begin{remark}
\label{rem:homotopy}
Homotopies do not in general commute with the maps in the corresponding $( \C ,S)$-factorizations. Thus, for the homotopy $s$ above, it need not be the case that $g_i \circ s_i = s_{i+1} \circ f_{i+1}$. However, as in the case of ordinary complexes, homotopies do commute with the \emph{square} of the maps in the factorizations: $g_{i+1} \circ g_i \circ s_i = s_{i+2} \circ f_{i+2} \circ f_{i+1}$. This follows from the defining equations combined with the fact that $\varphi$ and $\varphi'$ are morphisms.
\end{remark}

We shall now equip the homotopy category $\HFact_d ( \C,S,\eta )$ with the structure of a triangulated category, in terms of standard triangles defined in much the same way as for the homotopy category of complexes. The \emph{suspension} $\Sigma (M_i ,f_i)$ of a $( \C ,S)$-factorization $(M_i ,f_i)$ of $\eta$ is the sequence
\begin{center}
\begin{tikzpicture}
\diagram{d}{3em}{4em}{
M_2 & M_3 & \cdots & M_d & SM_1 & SM_2 \\
 };
\path[->, font = \scriptsize, auto]
(d-1-1) edge node{$-f_2$} (d-1-2)
(d-1-2) edge node{$-f_3$} (d-1-3)
(d-1-3) edge node{$-f_{d-1}$} (d-1-4)
(d-1-4) edge node{$-f_d$} (d-1-5)
(d-1-5) edge node{$-Sf_1$} (d-1-6);
\end{tikzpicture}
\end{center}
of morphisms in $\C$. This is again a $d$-fold $( \C, S)$-factorization of $\eta$. Namely, since $d$ is even, the signs cancel when $d$ morphisms are composed. Moreover, the equality $\eta_{SM} = S \eta_M$ guarantees that when $d$ consecutive (possibly suspended) morphisms are composed, the result is a component map of $\eta$. Next, let $\varphi$ be a morphism
\begin{center}
\begin{tikzpicture}
\diagram{d}{3em}{4em}{
M_1 & M_2 & \cdots & M_d & SM_1 \\
N_1 & N_2 & \cdots & N_d & SN_1 \\
 };
\path[->, font = \scriptsize, auto]
% upper horizontal maps
(d-1-1) edge node{$f_1$} (d-1-2)
(d-1-2) edge node{$f_2$} (d-1-3)
(d-1-3) edge node{$f_{d-1}$} (d-1-4)
(d-1-4) edge node{$f_d$} (d-1-5)
% lower horizontal maps
(d-2-1) edge node{$g_1$} (d-2-2)
(d-2-2) edge node{$g_2$} (d-2-3)
(d-2-3) edge node{$g_{d-1}$} (d-2-4)
(d-2-4) edge node{$g_d$} (d-2-5)
% vertical maps
(d-1-1) edge node{$\varphi_1$} (d-2-1)
(d-1-2) edge node{$\varphi_2$} (d-2-2)
(d-1-4) edge node{$\varphi_d$} (d-2-4)
(d-1-5) edge node{$S \varphi_1$} (d-2-5);
\end{tikzpicture}
\end{center}
in $\Fact_d ( \C,S, \eta )$. From the suspensions of the two $( \C, S)$-factorizations of $\eta$ involved, we obtain the commutative diagram
\begin{center}
\begin{tikzpicture}
\diagram{d}{3em}{4em}{
M_2 & M_3 & \cdots & M_d & SM_1 & SM_2 \\
N_2 & N_3 & \cdots & N_d & SN_1 & SN_2 \\
 };
\path[->, font = \scriptsize, auto]
% upper horizontal maps
(d-1-1) edge node{$-f_2$} (d-1-2)
(d-1-2) edge node{$-f_3$} (d-1-3)
(d-1-3) edge node{$-f_{d-1}$} (d-1-4)
(d-1-4) edge node{$-f_d$} (d-1-5)
(d-1-5) edge node{$-Sf_1$} (d-1-6)
% lower horizontal maps
(d-2-1) edge node{$-g_2$} (d-2-2)
(d-2-2) edge node{$-g_3$} (d-2-3)
(d-2-3) edge node{$-g_{d-1}$} (d-2-4)
(d-2-4) edge node{$-g_d$} (d-2-5)
(d-2-5) edge node{$-Sg_1$} (d-2-6)
% vertical maps
(d-1-1) edge node{$\varphi_2$} (d-2-1)
(d-1-2) edge node{$\varphi_3$} (d-2-2)
(d-1-4) edge node{$\varphi_d$} (d-2-4)
(d-1-5) edge node{$S \varphi_1$} (d-2-5)
(d-1-6) edge node{$S \varphi_2$} (d-2-6);
\end{tikzpicture}
\end{center}
displaying a morphism in $\Fact_d ( \C,S, \eta )$ from $\Sigma (M_i,f_i)$ to $\Sigma (N_i,g_i)$. This assignment is easily seen to respect homotopies: two homotopic morphisms from $(M_i,f_i)$ to $(N_i,g_i)$ give two homotopic morphisms from $\Sigma (M_i,f_i)$ to $\Sigma (N_i,g_i)$. Moreover, the suspension respects coproducts, addition and composition of morphisms. We therefore obtain an additive functor
\begin{center}
\begin{tikzpicture}
\diagram{d}{3em}{4em}{
\HFact_d \left ( \C,S, \eta \right ) & \HFact_d \left ( \C,S, \eta \right ) \\
 };
\path[->, font = \scriptsize, auto]
(d-1-1) edge node{$\Sigma$} (d-1-2);
\end{tikzpicture}
\end{center}
which is clearly an automorphism; its inverse $\Sigma^{-1}$ is given by right rotation.

Having defined the suspension $\Sigma$ on $\HFact_d ( \C,S, \eta )$, we next define mapping cones as in the classical case. Let $\varphi$ be the morphism above. Its \emph{mapping cone}, denoted $C_{\varphi}$, is the sequence
\begin{center}
\begin{tikzpicture}
\diagram{d}{3em}{4em}{
M_2 \oplus N_1 & M_3 \oplus N_2 & \cdots & SM_1 \oplus N_d & SM_2 \oplus SN_1 \\
 };
\path[->, font = \scriptsize, auto]
(d-1-1) edge node{$\left [ \begin{smallmatrix} -f_2 & 0 \\ \varphi_2 & g_1 \end{smallmatrix} \right ]$} (d-1-2)
(d-1-2) edge node{$\left [ \begin{smallmatrix} -f_3 & 0 \\ \varphi_3 & g_2 \end{smallmatrix} \right ]$} (d-1-3)
(d-1-3) edge node{$\left [ \begin{smallmatrix} -f_d & 0 \\ \varphi_d & g_{d-1} \end{smallmatrix} \right ]$} (d-1-4)
(d-1-4) edge node{$\left [ \begin{smallmatrix} -Sf_1 & 0 \\ S \varphi_1 & g_d \end{smallmatrix} \right ]$} (d-1-5);
\end{tikzpicture}
\end{center}
in $\C$. This is again a $d$-fold $( \C, S)$-factorization of $\eta$; the commuting squares in the morphism diagram ensure that the product of $d$ consecutive (possibly suspended) matrices is always a component map of $\eta$.

\begin{remark}
\label{rem:deven}
The construction of mapping cones is the main reason why the integer $d$ must be even. Consider the composition of the $d$ maps. For the cone to be a $( \C, S)$-factorization of $\eta$, this composition must equal $\eta_{M_2 \oplus N_1}$, or $\left [ \begin{smallmatrix} \eta_{M_2} & 0 \\ 0 & \eta_{N_1} \end{smallmatrix} \right ]$ in matrix notation. The even number of commuting squares in the morphism diagram guarantees that the lower left entry of this matrix is zero. If $d$ were odd, the composition would equal $\left [ \begin{smallmatrix} \eta_{M_2} & 0 \\ h & \eta_{N_1} \end{smallmatrix} \right ]$, with $h = S \varphi_1 \circ f_d \circ \cdots \circ f_2$. Thus, in this case, the composition would not equal $\eta_{M_2 \oplus N_1}$, in general.
\end{remark}

\sloppy Given a morphism $\varphi \colon (M_i,f_i) \to (N_i,g_i)$ as above, there is a natural morphism
\begin{center}
\begin{tikzpicture}
\diagram{d}{3em}{4em}{
(N_i,g_i) & C_{\varphi} \\
 };
\path[->, font = \scriptsize, auto]
(d-1-1) edge node{$i_{\varphi}$} (d-1-2);
\end{tikzpicture}
\end{center}
in $\Fact_d ( \C,S, \eta )$, displayed in the diagram
\begin{center}
\begin{tikzpicture}
\diagram{d}{3em}{4em}{
N_1 & N_2 & \cdots & N_d & SN_1 \\
M_2 \oplus N_1 & M_3 \oplus N_2 & \cdots & SM_1 \oplus N_d & SM_2 \oplus SN_1 \\
 };
\path[->, font = \scriptsize, auto]
% upper horizontal maps
(d-1-1) edge node{$g_1$} (d-1-2)
(d-1-2) edge node{$g_2$} (d-1-3)
(d-1-3) edge node{$g_{d-1}$} (d-1-4)
(d-1-4) edge node{$g_d$} (d-1-5)
% lower horizontal maps
(d-2-1) edge node{$\left [ \begin{smallmatrix} -f_2 & 0 \\ \varphi_2 & g_1 \end{smallmatrix} \right ]$} (d-2-2)
(d-2-2) edge node{$\left [ \begin{smallmatrix} -f_3 & 0 \\ \varphi_3 & g_2 \end{smallmatrix} \right ]$} (d-2-3)
(d-2-3) edge node{$\left [ \begin{smallmatrix} -f_d & 0 \\ \varphi_d & g_{d-1} \end{smallmatrix} \right ]$} (d-2-4)
(d-2-4) edge node{$\left [ \begin{smallmatrix} -Sf_1 & 0 \\ S \varphi_1 & g_d \end{smallmatrix} \right ]$} (d-2-5)
% vertical maps
(d-1-1) edge node{$\left [ \begin{smallmatrix} 0 \\ 1 \end{smallmatrix} \right ]$} (d-2-1)
(d-1-2) edge node{$\left [ \begin{smallmatrix} 0 \\ 1 \end{smallmatrix} \right ]$} (d-2-2)
(d-1-4) edge node{$\left [ \begin{smallmatrix} 0 \\ 1 \end{smallmatrix} \right ]$} (d-2-4)
(d-1-5) edge node{$\left [ \begin{smallmatrix} 0 \\ 1 \end{smallmatrix} \right ]$} (d-2-5);
\end{tikzpicture}
\end{center}
Similarly, there is a natural morphism
\begin{center}
\begin{tikzpicture}
\diagram{d}{3em}{4em}{
C_{\varphi} & \Sigma (M_i,f_i) \\
 };
\path[->, font = \scriptsize, auto]
(d-1-1) edge node{$\pi_{\varphi}$} (d-1-2);
\end{tikzpicture}
\end{center}
displayed in the diagram
\begin{center}
\begin{tikzpicture}
\diagram{d}{3em}{4em}{
M_2 \oplus N_1 & M_3 \oplus N_2 & \cdots & SM_1 \oplus N_d & SM_2 \oplus SN_1 \\
M_2 & M_3 & \cdots & SM_1 & SM_2 \\
 };
\path[->, font = \scriptsize, auto]
% upper horizontal maps
(d-1-1) edge node{$\left [ \begin{smallmatrix} -f_2 & 0 \\ \varphi_2 & g_1 \end{smallmatrix} \right ]$} (d-1-2)
(d-1-2) edge node{$\left [ \begin{smallmatrix} -f_3 & 0 \\ \varphi_3 & g_2 \end{smallmatrix} \right ]$} (d-1-3)
(d-1-3) edge node{$\left [ \begin{smallmatrix} -f_d & 0 \\ \varphi_d & g_{d-1} \end{smallmatrix} \right ]$} (d-1-4)
(d-1-4) edge node{$\left [ \begin{smallmatrix} -Sf_1 & 0 \\ S \varphi_1 & g_d \end{smallmatrix} \right ]$} (d-1-5)
% lower horizontal maps
(d-2-1) edge node{$-f_2$} (d-2-2)
(d-2-2) edge node{$-f_3$} (d-2-3)
(d-2-3) edge node{$-f_d$} (d-2-4)
(d-2-4) edge node{$-Sg_1$} (d-2-5)
% vertical maps
(d-1-1) edge node{$\left [ \begin{smallmatrix} 1 & 0 \end{smallmatrix} \right ]$} (d-2-1)
(d-1-2) edge node{$\left [ \begin{smallmatrix} 1 & 0 \end{smallmatrix} \right ]$} (d-2-2)
(d-1-4) edge node{$\left [ \begin{smallmatrix} 1 & 0 \end{smallmatrix} \right ]$} (d-2-4)
(d-1-5) edge node{$\left [ \begin{smallmatrix} 1 & 0 \end{smallmatrix} \right ]$} (d-2-5);
\end{tikzpicture}
\end{center}
We now follow normal procedure. Define a \emph{triangle} in the homotopy category $\HFact_d ( \C,S, \eta )$ to be a sequence
\begin{center}
\begin{tikzpicture}
\diagram{d}{3em}{4em}{
X & Y & Z & \Sigma X \\
 };
\path[->, font = \scriptsize, auto]
(d-1-1) edge node{$u$} (d-1-2)
(d-1-2) edge node{$v$} (d-1-3)
(d-1-3) edge node{$w$} (d-1-4);
\end{tikzpicture}
\end{center}
of objects and morphisms. A \emph{morphism} between two such triangles is a commutative diagram
\begin{center}
\begin{tikzpicture}
\diagram{d}{3em}{4em}{
X_1 & Y_1 & Z_1 & \Sigma X_1 \\
X_2 & Y_2 & Z_2 & \Sigma X_2 \\
 };
\path[->, font = \scriptsize, auto]
% upper horizontal maps
(d-1-1) edge node{$u_1$} (d-1-2)
(d-1-2) edge node{$v_1$} (d-1-3)
(d-1-3) edge node{$w_1$} (d-1-4)
% lower horizontal maps
(d-2-1) edge node{$u_2$} (d-2-2)
(d-2-2) edge node{$v_2$} (d-2-3)
(d-2-3) edge node{$w_2$} (d-2-4)
% vertical maps
(d-1-1) edge node{$\alpha$} (d-2-1)
(d-1-2) edge node{$\beta$} (d-2-2)
(d-1-3) edge node{$\gamma$} (d-2-3)
(d-1-4) edge node{$\Sigma \alpha$} (d-2-4);
\end{tikzpicture}
\end{center}
in $\HFact_d ( \C,S, \eta )$, and it is an \emph{isomorphism} of triangles if all the vertical morphisms are isomorphisms in $\HFact_d ( \C,S, \eta )$. For a morphism $\varphi \colon (M_i,f_i) \to (N_i,g_i)$ in $\Fact_d ( \C,S, \eta )$, we call the corresponding triangle
\begin{center}
\begin{tikzpicture}
\diagram{d}{3em}{4em}{
(M_i,f_i) & (N_i,g_i) & C_{\varphi} & \Sigma (M_i,f_i) \\
 };
\path[->, font = \scriptsize, auto]
(d-1-1) edge node{$[ \varphi ]$} (d-1-2)
(d-1-2) edge node{$[ i_{\varphi} ]$} (d-1-3)
(d-1-3) edge node{$[ \pi_{\varphi} ]$} (d-1-4);
\end{tikzpicture}
\end{center}
in $\HFact_d ( \C,S, \eta )$ a \emph{standard triangle}.

\begin{remark}
\label{rem:isomappingcones}
Suppose that $\varphi$ and $\varphi'$ are homotopic morphisms, with a homotopy $(s_1, \dots, s_d)$ as earlier in this section. Then the mapping cones $C_{\varphi}$ and $C_{\varphi'}$ are isomorphic in $\HFact_d ( \C,S, \eta )$, with an isomorphism $\lambda \colon C_{\varphi} \to C_{\varphi'}$ being given by $\lambda = \left ( \left [ \begin{smallmatrix} 1 & 0 \\ s_1 & 1 \end{smallmatrix} \right ], \dots,  \left [ \begin{smallmatrix} 1 & 0 \\ s_d & 1 \end{smallmatrix} \right ] \right )$. Its inverse is obtained by replacing the $s_i$ in the matrices by $-s_i$. Moreover, the diagram
\begin{center}
\begin{tikzpicture}
\diagram{d}{3em}{4em}{
(M_i,f_i) & (N_i,g_i) & C_{\varphi} & \Sigma (M_i,f_i) \\
(M_i,f_i) & (N_i,g_i) & C_{\varphi'} & \Sigma (M_i,f_i) \\
 };
\path[->, font = \scriptsize, auto]
% upper horizontal maps
(d-1-1) edge node{$[ \varphi ]$} (d-1-2)
(d-1-2) edge node{$[ i_{\varphi} ]$} (d-1-3)
(d-1-3) edge node{$[ \pi_{\varphi} ]$} (d-1-4)
% lower horizontal maps
(d-2-1) edge node{$[ \varphi' ]$} (d-2-2)
(d-2-2) edge node{$[ i_{\varphi'} ]$} (d-2-3)
(d-2-3) edge node{$[ \pi_{\varphi'} ]$} (d-2-4)
% vertical maps
(d-1-1) edge node{$1$} (d-2-1)
(d-1-2) edge node{$1$} (d-2-2)
(d-1-3) edge node{$[ \lambda ]$} (d-2-3)
(d-1-4) edge node{$1$} (d-2-4);
\end{tikzpicture}
\end{center}
commutes in $\HFact_d ( \C,S, \eta )$; in fact, the equalities $i_{\varphi'} = \lambda \circ i_{\varphi}$ and $\pi_{\varphi'} \circ \lambda = \pi_{\varphi}$ hold already in $\Fact_d ( \C,S, \eta )$. Thus, up to isomorphism of triangles in $\HFact_d ( \C,S, \eta )$, a standard triangle is independent of the representative chosen for its base morphism $[ \varphi ]$.
\end{remark}

The following result shows that the homotopy category $\HFact_d ( \C,S, \eta )$, together with its suspension $\Sigma$ and the collection of triangles that are isomorphic to standard triangles, is a triangulated category. For the proof, one could adapt the classical proof for the homotopy category of complexes over an additive category, with some modifications. However, we shall instead prove that $\HFact_d ( \C,S, \eta )$ is an \emph{algebraic} triangulated category. Recall that a triangulated category is algebraic if it can be characterized -- up to equivalence of triangulated categories --  in one of the following equivalent ways:
\begin{itemize}
\item[(1)] as the stable category of some Frobenius exact category;
\item[(2)] as a full triangulated subcategory of the homotopy category of some additive category;
\item[(3)] as the (zeroth) cohomology category of some pretriangulated differential graded category.
\end{itemize}
For the equivalence of these three characterizations of a triangulated category, we refer to \cite{BondalKapranov, Keller, Krause, Schwede}. We shall use the approach via differential graded (DG) categories, by constructing a DG enhancement of $\HFact_d ( \C,S, \eta )$.

\begin{theorem}\label{thm:main}
Let $( \C, S )$ be a suspended additive category, and $\eta \colon 1_{\C} \to S$ a natural transformation with the property that $\eta_{SM} = S \eta_M$ for every object $M$ of $\C$. Furthermore, let $d \ge 2$ be an even integer, $\HFact_d ( \C,S, \eta )$ the homotopy category of $d$-fold $( \C, S )$-factorizations of $\eta$, and $\Sigma \colon \HFact_d ( \C,S, \eta ) \to \HFact_d ( \C,S, \eta )$ the suspension constructed above. Finally, let $\Delta$ be the collection of all triangles in $\HFact_d ( \C,S, \eta )$ isomorphic to a standard triangle. Then 
$$\left ( \HFact_d \left ( \C,S, \eta \right ), \Sigma, \Delta \right )$$
is an algebraic triangulated category.
\end{theorem}

\begin{proof}
We construct a DG category $\D$ as follows. The objects are the same as in $\Fact_d ( \C,S, \eta )$, that is, all the $d$-fold $( \C, S )$-factorizations of $\eta$. For two objects $M = (M_i,f_i)$ and $N = (N_i,g_i)$, and an integer $n \in \mathbb{Z}$, we define $\Hom_{\D}^n(M,N)$ as the set of all $d$-tuples $\varphi = ( \varphi_1, \dots, \varphi_d )$ of morphisms $\varphi_i \colon M_i \to N_{i+n}$, with the property that each \emph{double square} in the diagram
\begin{center}
\begin{tikzpicture}
\diagram{d}{3em}{4em}{
M_1 & M_2 & M_3 & \cdots & M_d & SM_1 \\
N_{1+n} & N_{2+n} & N_{3+n} & \cdots & N_{d+n} & SN_{1+n} \\
 };
\path[->, font = \scriptsize, auto]
% upper horizontal maps
(d-1-1) edge node{$f_1$} (d-1-2)
(d-1-2) edge node{$f_2$} (d-1-3)
(d-1-3) edge node{$f_3$} (d-1-4)
(d-1-4) edge node{$f_{d-1}$} (d-1-5)
(d-1-5) edge node{$f_d$} (d-1-6)
% lower horizontal maps
(d-2-1) edge node{$g_{1+n}$} (d-2-2)
(d-2-2) edge node{$g_{2+n}$} (d-2-3)
(d-2-3) edge node{$g_{3+n}$} (d-2-4)
(d-2-4) edge node{$g_{d-1+n}$} (d-2-5)
(d-2-5) edge node{$g_{d+n}$} (d-2-6)
% vertical maps
(d-1-1) edge node{$\varphi_1$} (d-2-1)
(d-1-2) edge node{$\varphi_2$} (d-2-2)
(d-1-3) edge node{$\varphi_3$} (d-2-3)
(d-1-5) edge node{$\varphi_d$} (d-2-5)
(d-1-6) edge node{$S \varphi_1$} (d-2-6);
\end{tikzpicture}
\end{center}
commutes: 
$$\varphi_{i+2} \circ f_{i+1} \circ f_i = g_{i+1+n} \circ g_{i+n} \circ \varphi_i$$
The indices are here taken modulo $d$, in the sense that if $m = qd + r$ with $1 \le r \le d$, then $N_m = S^qN_r$ and $g_m = S^qg_r$.

We now turn the graded $\mathbb{Z}$-module 
$$\Hom_{\D}^*(M,N) = \bigoplus_{n \in \mathbb{Z}} \Hom_{\D}^n(M,N)$$
into a DG $\mathbb{Z}$-module in the usual way: for the element $\varphi \in \Hom_{\D}^n(M,N)$ above we define
$$\partial ( \varphi ) = g \circ \varphi + (-1)^{n+1} \varphi \circ f$$
Thus $\partial ( \varphi )$ is the element in $\Hom_{\D}^{n+1}(M,N)$ displayed in the diagram
\begin{center}
\begin{tikzpicture}
\diagram{d}{3em}{4em}{
M_1 & M_2 & M_3 & \cdots & M_d & SM_1 \\
N_{2+n} & N_{3+n} & N_{4+n} & \cdots & SN_{1+n} & SN_{2+n} \\
 };
\path[->, font = \scriptsize, auto]
% upper horizontal maps
(d-1-1) edge node{$f_1$} (d-1-2)
(d-1-2) edge node{$f_2$} (d-1-3)
(d-1-3) edge node{$f_3$} (d-1-4)
(d-1-4) edge node{$f_{d-1}$} (d-1-5)
(d-1-5) edge node{$f_d$} (d-1-6)
% lower horizontal maps
(d-2-1) edge node{$g_{2+n}$} (d-2-2)
(d-2-2) edge node{$g_{3+n}$} (d-2-3)
(d-2-3) edge node{$g_{4+n}$} (d-2-4)
(d-2-4) edge node{$g_{d+n}$} (d-2-5)
(d-2-5) edge node{$Sg_{1+n}$} (d-2-6)
% vertical maps
(d-1-1) edge node{$\partial ( \varphi_1 )$} (d-2-1)
(d-1-2) edge node{$\partial ( \varphi_2 )$} (d-2-2)
(d-1-3) edge node{$\partial ( \varphi_3 )$} (d-2-3)
(d-1-5) edge node{$\partial ( \varphi_d )$} (d-2-5)
(d-1-6) edge node{$\partial ( S \varphi_1 )$} (d-2-6);
\end{tikzpicture}
\end{center}
with $\partial ( \varphi_i ) = g_{i+n} \circ \varphi_i + (-1)^{n+1} \varphi_{i+1} \circ f_i$. The requirement that each double square in the $\varphi$-diagram commutes now gives
\begin{eqnarray*}
\partial^2 ( \varphi ) & = & \partial \left ( g \circ \varphi + (-1)^{n+1} \varphi \circ f \right ) \\
& = & g \circ \left ( g \circ \varphi + (-1)^{n+1} \varphi \circ f \right ) + (-1)^{n+2} \left ( g \circ \varphi + (-1)^{n+1} \varphi \circ f \right ) \circ f \\
& = & g \circ g \circ \varphi - \varphi \circ f \circ f \\
& = & 0
\end{eqnarray*}
showing that $\left ( \Hom_{\D}^*(M,N), \partial \right )$ is a cochain complex, that is, a DG $\mathbb{Z}$-module.

Compositions of morphisms in $\D$ are morphisms
\begin{center}
\begin{tikzpicture}
\diagram{d}{3em}{4em}{
\Hom_{\D}^*(N,L) \otimes_{\mathbb{Z}} \Hom_{\D}^*(M,N) & \Hom_{\D}^*(M,L) \\
 };
\path[->, font = \scriptsize, auto]
(d-1-1) edge (d-1-2);
\end{tikzpicture}
\end{center}
of complexes. Furthermore, for each object $M \in \D$, the identity morphism $1_M \in \Hom_{\D}^0(M,M)$ satisfies $\partial ( 1_M ) =0$. Thus $\D$ is a DG category; we must show that it is pretriangulated.

As above, consider two objects $M = (M_i,f_i)$ and $N = (N_i,g_i)$ in $\D$. For an integer $t \in \mathbb{Z}$, the shifted DG $\mathbb{Z}$-module $\Hom_{\D}^*(M,N)[t]$ is the cochain complex with $\Hom_{\D}^{n+t}(M,N)$ in degree $n$, and with differential $(-1)^t \partial$. The right $\D$-module $\Hom_{\D}^*(-,N)[t]$ -- that is, the contravariant DG functor
\begin{center}
\begin{tikzpicture}
\diagram{d}{3em}{4em}{
M & \Hom_{\D}^*(M,N)[t] \\
 };
\path[->, font = \scriptsize, auto]
(d-1-1) edge[|->] (d-1-2);
\end{tikzpicture}
\end{center}
from $\D$ to the category of DG $\mathbb{Z}$-modules -- is representable: one checks that it is isomorphic to the right $\D$-module $\Hom_{\D}^*(-, \Sigma^tN)$, where $\Sigma$ is the suspension in $\Fact_d ( \C,S, \eta )$. Next, let $\varphi \in Z^0 \left ( \Hom_{\D}^*(M,N), \partial \right )$, that is, $\varphi \in \Hom_{\D}^0(M,N)$ and $0 = \partial ( \varphi )= g \circ \varphi - \varphi \circ f$. Then each square in the diagram
\begin{center}
\begin{tikzpicture}
\diagram{d}{3em}{4em}{
M_1 & M_2 & M_3 & \cdots & M_d & SM_1 \\
N_{1} & N_{2} & N_{3} & \cdots & N_{d} & SN_{1} \\
 };
\path[->, font = \scriptsize, auto]
% upper horizontal maps
(d-1-1) edge node{$f_1$} (d-1-2)
(d-1-2) edge node{$f_2$} (d-1-3)
(d-1-3) edge node{$f_3$} (d-1-4)
(d-1-4) edge node{$f_{d-1}$} (d-1-5)
(d-1-5) edge node{$f_d$} (d-1-6)
% lower horizontal maps
(d-2-1) edge node{$g_{1}$} (d-2-2)
(d-2-2) edge node{$g_{2}$} (d-2-3)
(d-2-3) edge node{$g_{3}$} (d-2-4)
(d-2-4) edge node{$g_{d-1}$} (d-2-5)
(d-2-5) edge node{$g_{d+n}$} (d-2-6)
% vertical maps
(d-1-1) edge node{$\varphi_1$} (d-2-1)
(d-1-2) edge node{$\varphi_2$} (d-2-2)
(d-1-3) edge node{$\varphi_3$} (d-2-3)
(d-1-5) edge node{$\varphi_d$} (d-2-5)
(d-1-6) edge node{$S \varphi_1$} (d-2-6);
\end{tikzpicture}
\end{center}
commutes, so that $\varphi$ is a morphism in $\Fact_d ( \C,S, \eta )$. It induces a morphism 
\begin{center}
\begin{tikzpicture}
\diagram{d}{3em}{4em}{
\Hom_{\D}^*(-,M) & \Hom_{\D}^*(-,N) \\
 };
\path[->, font = \scriptsize, auto]
(d-1-1) edge node{$\widehat{\varphi}$} (d-1-2);
\end{tikzpicture}
\end{center}
of right $\D$-modules, and the mapping cone $C( \widehat{\varphi} )$ of the latter is the right $\D$-module defined as follows. For an object $L \in \D$, the graded $\mathbb{Z}$-module $C( \widehat{\varphi} )^*(L)$ has
$$\Hom_{\D}^{n+1}(L,M) \oplus \Hom_{\D}^n(L,N)$$
in degree $n$, with differential $C( \widehat{\varphi} )^n(L) \to C( \widehat{\varphi} )^{n+1}(L)$ given by
\begin{center}
\begin{tikzpicture}
\diagram{d}{3em}{4em}{
(\theta, \psi ) & \left ( - \partial ( \theta ), \varphi \circ \theta + \partial ( \psi ) \right ) \\
 };
\path[->, font = \scriptsize, auto]
(d-1-1) edge[|->] (d-1-2);
\end{tikzpicture}
\end{center}
This turns $C( \widehat{\varphi} )^*(L)$ into a cochain complex, and one now checks that $C( \widehat{\varphi} )$ is representable: it is isomorphic to the right $\D$-module $\Hom_{\D}^*(-, C_{\varphi} )$, where $C_{\varphi}$ is the mapping cone of $\varphi$ in $\Fact_d ( \C,S, \eta )$. Consequently, the DG category $\D$ is pretriangulated.

Since $\D$ is pretriangulated, its zeroth cohomology $\Ho^0 ( \D )$ is triangulated. Moreover, as for the homotopy category of complexes over an additive category, it now follows from the definition of $\Ho^0 ( \D )$ that it is equivalent, as a triangulated category, to $\left ( \HFact_d ( \C,S, \eta ), \Sigma, \Delta \right )$; for details, we refer to \cite[Section 2]{Schwede}, keeping in mind that we use cohomological notation. Namely, for all objects $M,N \in \D$ and morphisms $\varphi \in Z^0 \left ( \Hom_{\D}^*(M,N), \partial \right )$ -- in other words, morphisms in $\Fact_d ( \C,S, \eta )$ -- the above representability properties in $\D$ of the mapping cone $C_{\varphi}$ provide universal morphisms $i_{\varphi} \colon N \to C_{\varphi}$ and $\pi_{\varphi} \colon C_{\varphi} \to \Sigma M$ in $Z^0 \left ( \Hom_{\D}^*(M,N), \partial \right )$, and these are precisely the morphisms in $\Fact_d ( \C,S, \eta )$ we denoted the same way. The distinguished triangles in $\Ho^0 ( \D )$ are now the ones that are isomorphic to the images of the triangles in $Z^0 \left ( \Hom_{\D}^*(M,N), \partial \right )$ of the form 
\begin{center}
\begin{tikzpicture}
\diagram{d}{3em}{4em}{
M & N & C_{\varphi} & \Sigma M \\
 };
\path[->, font = \scriptsize, auto]
(d-1-1) edge node{$\varphi$} (d-1-2)
(d-1-2) edge node{$i_{\varphi}$} (d-1-3)
(d-1-3) edge node{$\pi_{\varphi}$} (d-1-4);
\end{tikzpicture}
\end{center}
However, the morphisms in $\Ho^0 ( \D )$ are precisely the homotopy equivalence classes of morphisms in $\Fact_d ( \C,S, \eta )$: a morphism $\theta \colon M \to N$ in $Z^0 \left ( \Hom_{\D}^*(M,N), \partial \right )$ is null-homotopic if and only if it is the image under the differential $\partial$ on $\Hom_{\D}^*(M,N)$ of an element $s \in \Hom_{\D}^{-1}(M,N)$. This follows from the definition of $\Hom_{\D}^{-1}(M,N)$, together with Remark \ref{rem:homotopy}. Consequently, the distinguished triangles in $\Ho^0 ( \D )$ are those that are isomorphic to triangles of the form 
\begin{center}
\begin{tikzpicture}
\diagram{d}{3em}{4em}{
M & N & C_{\varphi} & \Sigma M \\
 };
\path[->, font = \scriptsize, auto]
(d-1-1) edge node{$[ \varphi ]$} (d-1-2)
(d-1-2) edge node{$[ i_{\varphi} ]$} (d-1-3)
(d-1-3) edge node{$[ \pi_{\varphi} ]$} (d-1-4);
\end{tikzpicture}
\end{center}
This shows that the category $\D$ is a DG enhancement of $\HFact_d ( \C,S, \eta )$.
\end{proof}

Let us now look at a few examples.

\begin{example}\label{ex:periodiccomplexes}
When the natural transformation $\eta \colon 1_{\C} \to S$ is the zero transformation, then we obtain the homotopy category $\HFact_d ( \C,S, 0 )$. The objects are sequences
\begin{center}
\begin{tikzpicture}
\diagram{d}{3em}{4em}{
M_1 & M_2 & \cdots & M_d & SM_1 \\
 };
\path[->, font = \scriptsize, auto]
(d-1-1) edge node{$f_1$} (d-1-2)
(d-1-2) edge node{$f_2$} (d-1-3)
(d-1-3) edge node{$f_{d-1}$} (d-1-4)
(d-1-4) edge node{$f_d$} (d-1-5);
\end{tikzpicture}
\end{center}
in $\C$, in which all the $d$-fold compositions
$$f_d \circ f_{d-1} \circ \cdots \circ f_1, \hspace{3mm} Sf_1 \circ f_d \circ \cdots \circ f_2, \hspace{3mm} \dots, \hspace{3mm}
Sf_{d-1} \circ \cdots \circ Sf_1 \circ f_d$$
are zero. 

The triangulated category $\left ( \HFact_d ( \C,S, 0 ), \Sigma, \Delta \right )$ is equivalent to the homotopy category of $S$-periodic $d$-complexes. The objects of the latter are all sequences
\begin{center}
\begin{tikzpicture}
\diagram{d}{3em}{2.8em}{
\cdots & S^{-1}M_d & M_1 & M_2 & \cdots & M_d & SM_1 & \cdots \\
 };
\path[->, font = \scriptsize, auto]
(d-1-1) edge node{$S^{-1}f_{d-1}$} (d-1-2)
(d-1-2) edge node{$S^{-1}f_d$} (d-1-3)
(d-1-3) edge node{$f_1$} (d-1-4)
(d-1-4) edge node{$f_2$} (d-1-5)
(d-1-5) edge node{$f_{d-1}$} (d-1-6)
(d-1-6) edge node{$f_d$} (d-1-7)
(d-1-7) edge node{$Sf_1$} (d-1-8);
\end{tikzpicture}
\end{center}
in which the composition of $d$ consecutive maps is zero, and where also the homotopies are $d$-$S$-periodic. The suspension is the (left) shift with a sign change on the maps, and the distinguished triangles are defined in terms of mapping cones. In particular, when $d = 2$, we obtain the homotopy category of complexes of the form
\begin{center}
\begin{tikzpicture}
\diagram{d}{3em}{4em}{
\cdots & S^{-1}M_2 & M_1 & M_2 & SM_1 & \cdots \\
 };
\path[->, font = \scriptsize, auto]
(d-1-1) edge node{$S^{-1}f_1$} (d-1-2)
(d-1-2) edge node{$S^{-1}f_2$} (d-1-3)
(d-1-3) edge node{$f_1$} (d-1-4)
(d-1-4) edge node{$f_2$} (d-1-5)
(d-1-5) edge node{$Sf_1$} (d-1-6);
\end{tikzpicture}
\end{center}
with $2$-$S$-periodic homotopies. Specializing further, by taking $S = 1_{\C}$, we obtain the homotopy category of $2$-periodic complexes 
\begin{center}
\begin{tikzpicture}
\diagram{d}{3em}{4em}{
\cdots & M_1 & M_2 & M_1 & \cdots \\
 };
\path[->, font = \scriptsize, auto]
(d-1-1) edge node{$f_2$} (d-1-2)
(d-1-2) edge node{$f_1$} (d-1-3)
(d-1-3) edge node{$f_2$} (d-1-4)
(d-1-4) edge node{$f_1$} (d-1-5);
\end{tikzpicture}
\end{center}
and $2$-periodic homotopies; it is equivalent to $\left ( \HFact_2 ( \C,1_{\C}, 0 ), \Sigma, \Delta \right )$. 

Generalized complexes as above were introduced by Kapranov in \cite{Kapranov}, using the terminology \emph{$N$-complexes}. Homotopy categories of such complexes were then studied in \cite{GillespieHovey} and \cite{IyamaKatoMiyachi}, using a modified version of homotopies.
\end{example}

\begin{example}\label{ex:matrixfactorizations}
The origin of the topic of this paper is the theory of matrix factorizations for commutative rings, introduced by Eisenbud in \cite{Eisenbud}. Thus let $R$ be a commutative ring, and $x$ an element of $R$. Take as $\C$ the category $\Proj (R)$ of finitely generated projective $R$-modules, and as suspension $S$ the identity automorphism $1_{\Proj (R)}$. Finally, as the natural transformation $1_{\Proj (R)} \to 1_{\Proj (R)}$ we take the multiplication transformation $\eta_x$ induced by $x$, and set $d=2$. Then the objects of the homotopy category $\HFact_2 ( \Proj (R), 1_{\Proj (R)}, \eta_x )$ are diagrams
\begin{center}
\begin{tikzpicture}
\diagram{d}{3em}{4em}{
P & Q & P \\
 };
\path[->, font = \scriptsize, auto]
(d-1-1) edge node{$f_1$} (d-1-2)
(d-1-2) edge node{$f_2$} (d-1-3);
\end{tikzpicture}
\end{center}
in which $P$ and $Q$ are finitely generated projective $R$-modules, and where the maps satisfy the equalities $f_2 \circ f_1 = x \cdot 1_P$ and $f_1 \circ f_2 = x \cdot 1_Q$. If we specialize further by taking $\C$ to be the category $\Free (R)$ of finitely generated \emph{free} $R$-modules, then by choosing bases for the modules the maps are given in terms of square matrices which factorize the diagonal matrix for $x$.

Suppose that the element $x$ is a non-zerodivisor in $R$, and denote the quotient $R/(x)$ by $Q$. When we reduce a matrix factorization modulo $x$ and extend in both directions, we obtain a $2$-periodic acyclic complex of finitely generated projective $Q$-modules. As shown in \cite{BerghJorgensen}, this assignment induces a fully faithful triangle functor from $\left ( \HFact_2 ( \Proj (R), 1_{\Proj (R)}, \eta_x ), \Sigma, \Delta \right )$ to the homotopy category of totally acyclic complexes of finitely generated projective $Q$-modules. When $R$ is a regular local ring, then the latter is equivalent to the singularity category of the hypersurface ring $Q$, that is, the Verdier quotient of the bounded derived category of $Q$-modules by the perfect complexes (and also equivalent to the stable category of maximal Cohen-Macaulay $Q$-modules). In this case, reduction modulo $x$ actually induces an \emph{equivalence} between $\left ( \HFact_2 ( \Proj (R), 1_{\Proj (R)}, \eta_x ), \Sigma, \Delta \right )$ and the singularity category of $Q$; this was observed by Buchweitz in \cite{Buchweitz}, and proved explicitly by Orlov in \cite{Orlov}. In recent years, there have been numerous generalizations of this result, in many different directions. Also, matrix factorizations with more factors were recently introduced in \cite{Tribone}, including homotopy categories of such, using the modified version of homotopies mentioned at the end of Example \ref{ex:periodiccomplexes}.
\end{example}

\begin{example}\label{ex:twistedmatrixfactorizations}
Let now $B$ be a possibly noncommutative ring, $w \in B$ a central element, and denote the quotient $B/(w)$ by $A$. Furthermore, suppose that $\nu \colon B \to B$ is an automorphism satisfying $\nu (wr) = wr$ for all $r \in B$. In \cite{BerghErdmann}, certain twisted matrix factorizations over $B$ were introduced (see also  \cite{CassidyEtAl} for a related graded concept), and these give rise to complexes of free modules over $A$, as in the classical case.

Namely, denote the category of finitely generated free left $B$-modules by $\Free (B)$. Given such a module $F$, its twist ${_{\nu}F}$ is again free, and the assignment $F \mapsto {_{\nu}F}$ induces an additive automorphism $S_{\nu}$ on $\Free (B)$. Furthermore, as explained in \cite[Section 3]{BerghErdmann}, the map $\eta_F \colon F \to S_{\nu}F$ given by $m \mapsto wm$ is $B$-linear, and the collection $\eta_w = \left \{ \eta_F \mid F \in \Free (B) \right \}$ forms a natural transformation $1_{\Free (B)} \to S_{\nu}$ with $\eta_{S_{\nu}F} = S_{\nu} \eta_F$ for all $F$. We may therefore form the homotopy category $\HFact_2 ( \Free(B), S_{\nu}, \eta_w )$, whose objects are diagrams
\begin{center}
\begin{tikzpicture}
\diagram{d}{3em}{4em}{
F & G & S_{\nu}F \\
 };
\path[->, font = \scriptsize, auto]
(d-1-1) edge node{$f_1$} (d-1-2)
(d-1-2) edge node{$f_2$} (d-1-3);
\end{tikzpicture}
\end{center}
of free $B$-modules, and where the maps satisfy the equalities $f_2 \circ f_1 = \eta_F$ and $f_1 \circ f_2 = \eta_G$ (in the latter equality, the map $f_1$ should strictly speaking be $S_{\nu}f_1$, but the automorphism $S_{\nu}$ is the identity on maps). By \cite[Proposition 3.2]{BerghErdmann}, reduction modulo $w$ induces a triangle functor from $\left ( \HFact_2 ( \Free(B), S_{\nu}, \eta_w ), \Sigma, \Delta \right )$ to the homotopy category of complexes of finitely generated free left $A$-modules. The image of the factorization displayed above is the complex
\begin{center}
\begin{tikzpicture}
\diagram{d}{3em}{2.5em}{
\cdots & {_{\nu^{n-1}}(G/wG)} & {_{\nu^n}(F/wF)} & {_{\nu^n}(G/wG)} & {_{\nu^{n+1}}(F/wF)} & \cdots \\
 };
\path[->, font = \scriptsize, auto]
(d-1-1) edge (d-1-2)
(d-1-2) edge node{$\overline{f_2}$} (d-1-3)
(d-1-3) edge node{$\overline{f_1}$} (d-1-4)
(d-1-4) edge node{$\overline{f_2}$} (d-1-5)
(d-1-5) edge (d-1-6);
\end{tikzpicture}
\end{center}
with the free $A$-module $F/wF$ in degree zero. Note that since the automorphism $\nu$ satisfies $\nu (wr) = wr$ for all $r \in B$, it induces an automorphism on the quotient ring $A$; we have denoted this by $\nu$ as well.

A specific example was studied in detail in \cite[Section 4]{BerghErdmann}. Let $k$ be a field, $B$ the $k$-algebra
$$k \langle x,y \rangle / ( x^2, y^2, xyx, yxy )$$
and take $w = xy-qyx$ for some nonzero $q \in k$. Furthermore, consider the automorphism $\nu$ on $B$ defined by $x \mapsto -q^{-1}x$ and $y \mapsto -qy$; it trivially satisfies $\nu (wr) = wr$ for all $r \in B$. Then by \cite[Theorem 4.2]{BerghErdmann}, reduction modulo $w$ induces a triangle functor from $\left ( \HFact_2 ( \Free(B), S_{\nu}, \eta_w ), \Sigma, \Delta \right )$ to the homotopy category of \emph{acyclic} complexes of finitely generated free $A$-modules. Consequently, even though the element $w$ is far from being a regular element in $B$, this result is analogous to the classical case in Example \ref{ex:matrixfactorizations}.

The algebra $A$ is the four-dimensional quantum complete intersection
$$k \langle x,y \rangle / ( x^2, xy-qyx, y^2 )$$
It is local and selfinjective, and so the homotopy category of acyclic complexes of finitely generated free left $A$-modules is therefore equivalent to the stable module category $\stmod A$ of finitely generated left $A$-modules. Thus reduction modulo $w$ induces a triangle functor
\begin{center}
\begin{tikzpicture}
\diagram{d}{3em}{4em}{
\left ( \HFact_2 \left ( \Free(B), S_{\nu}, \eta_w \right ), \Sigma, \Delta \right ) & \stmod A \\
 };
\path[->, font = \scriptsize, auto]
(d-1-1) edge (d-1-2);
\end{tikzpicture}
\end{center}
From the construction of the functor, we see that the image is contained in the thick subcategory $\left ( \stmod A \right )_{\cx \le 1}$, formed by the $A$-modules of complexity at most one, that is, the modules for which there are bounds on the dimensions of the terms in the minimal projective resolutions. As shown in \cite[Theorem 4.6]{BerghErdmann}, when the field $k$ is algebraically closed, then the triangle functor
\begin{center}
\begin{tikzpicture}
\diagram{d}{3em}{4em}{
\left ( \HFact_2 \left ( \Free(B), S_{\nu}, \eta_w \right ), \Sigma, \Delta \right ) & \left ( \stmod A \right )_{\cx \le 1} \\
 };
\path[->, font = \scriptsize, auto]
(d-1-1) edge (d-1-2);
\end{tikzpicture}
\end{center}
is almost essentially surjective (i.e.\ dense); namely, all the indecomposable modules of $\left ( \stmod A \right )_{\cx \le 1}$ (and there are infinitely many of them) except two special ones lie in the image of the functor.
\end{example}

\begin{example}\label{ex:BallardEtAl}
When the category $\C$ is abelian and $d=2$, then the homotopy category $\HFact_2 ( \C,S, \eta )$ coincides with the homotopy factorization category introduced by Ballard et al.\ in \cite{BallardEtAl}. They introduce resolutions of such factorizations, and use these to obtain a number of interesting results. For example, given a smooth projective hypersurface, they establish a derived equivalence to the homotopy category of matrix factorizations over a certain noncommutative algebra.
\end{example}

\section{Totally acyclic complexes}
\label{sec:tac}

In this section we explore some triangulated functors from the homotopy category of factorizations to homotopy categories over rings. We keep the same notation as in Section \ref{sec:main}: let $( \C, S )$ be a suspended additive category, and $\eta \colon 1_{\C} \to S$ a natural transformation commuting with $S$, in the sense that $\eta_{SM} = S \eta_M$ for every object $M$ of $\C$. 

Fix an object $M \in \C$. From every object $(M_i,f_i) \in \Fact_d ( \C,S,\eta )$, that is, every $d$-fold $( \C, S )$-factorization
\begin{center}
\begin{tikzpicture}
\diagram{d}{3em}{4em}{
M_1 & M_2 & \cdots & M_d & SM_1 \\
 };
\path[->, font = \scriptsize, auto]
(d-1-1) edge node{$f_1$} (d-1-2)
(d-1-2) edge node{$f_2$} (d-1-3)
(d-1-3) edge node{$f_{d-1}$} (d-1-4)
(d-1-4) edge node{$f_d$} (d-1-5);
\end{tikzpicture}
\end{center}
of $\eta$, we obtain a (doubly infinite) sequence
\begin{center}
\begin{tikzpicture}
\diagram{d}{3em}{3em}{
\cdots & \Hom_{\C} \left ( M,M_1 \right ) & \cdots & \Hom_{\C} \left ( M,M_d \right ) & \cdots \\
 };
\path[->, font = \scriptsize, auto]
(d-1-1) edge node{$(S^{-1}f_d)_*$} (d-1-2)
(d-1-2) edge node{$(f_1)_*$} (d-1-3)
(d-1-3) edge node{$(f_{d-1})_*$} (d-1-4)
(d-1-4) edge node{$(f_d)_*$} (d-1-5);
\end{tikzpicture}
\end{center}
of abelian groups. This is a sequence of right $\Hom_{\C}(M,M)$-modules. 

Now let $\A$ be any additive category, and consider the homotopy category $\HSeq \A$ of all sequences over $\A$. Thus the objects are sequences 
\begin{center}
\begin{tikzpicture}
\diagram{d}{3em}{4em}{
\cdots & A_{n+1} & A_n & A_{n-1} & \cdots \\
 };
\path[->, font = \scriptsize, auto]
(d-1-1) edge (d-1-2)
(d-1-2) edge node{$d_{n+1}$} (d-1-3)
(d-1-3) edge node{$d_n$} (d-1-4)
(d-1-4) edge (d-1-5);
\end{tikzpicture}
\end{center}
of objects and morphisms in $\A$, and the morphisms are equivalence classes of chain maps. Similarly, given a positive even integer $t \ge 2$, we denote by $\K_t \A$ the homotopy category of all $t$-complexes, that is, sequences in which each $t$-fold composition is zero. These are both triangulated categories with the usual shift functor as suspension, and distinguished triangles defined in terms of mapping cones and standard triangles (the reason why we need $t$ to be even for $\K_t \A$ is the same as in Remark \ref{rem:deven}). The standard proof showing that the homotopy category of complexes is triangulated carries over verbatim; one can also replace $\HFact_d ( \C,S,\eta )$ with $\HSeq \A$ or $\K_t \A$ in the proof of Theorem \ref{thm:main}, and so the latter are algebraic triangulated categories. Of course, the category $\K_2 \A$ is just the usual homotopy category $\K \A$ of ordinary complexes. For a ring $R$, we denote by $\HSeq R$ and $\K_t R$ the homotopy categories of sequences and $t$-complexes of right $R$-modules.

\begin{theorem}\label{thm:functors}
Let $( \C, S )$ be a suspended additive category, and $\eta \colon 1_{\C} \to S$ a natural transformation with $\eta_{SM} = S \eta_M$ for every object $M$ of $\C$. Furthermore, let $d \ge 2$ be an even integer, and $\left ( \HFact_d ( \C,S, \eta ), \Sigma, \Delta \right )$ the triangulated homotopy category of factorizations from \emph{Theorem \ref{thm:main}}. Finally, fix an object $M \in \C$, and denote the endomorphism ring $\Hom_{\C}(M,M)$ by $\Gamma_M$.

\emph{(1)} The functor $\Hom_{\C}(M,-)$ induces a triangle functor
\begin{center}
\begin{tikzpicture}
\diagram{d}{3em}{4em}{
\left ( \HFact_d \left ( \C,S, \eta \right ), \Sigma, \Delta \right ) & \HSeq \Gamma_M \\
 };
\path[->, font = \scriptsize, auto]
(d-1-1) edge (d-1-2);
\end{tikzpicture}
\end{center}
It maps a $d$-fold factorization $(M_i,f_i) \in \HFact_d ( \C,S, \eta )$ to the sequence
\begin{center}
\begin{tikzpicture}
\diagram{d}{3em}{3em}{
\cdots & \Hom_{\C} \left ( M,M_1 \right ) & \cdots & \Hom_{\C} \left ( M,M_d \right ) & \cdots \\
 };
\path[->, font = \scriptsize, auto]
(d-1-1) edge node{$(S^{-1}f_d)_*$} (d-1-2)
(d-1-2) edge node{$(f_1)_*$} (d-1-3)
(d-1-3) edge node{$(f_{d-1})_*$} (d-1-4)
(d-1-4) edge node{$(f_d)_*$} (d-1-5);
\end{tikzpicture}
\end{center}
of right $\Gamma_M$-modules, and a homotopy equivalence class of morphisms of $d$-fold factorizations to the corresponding homotopy equivalence class of morphisms of sequences.

\emph{(2)} Suppose that $f$ is a central element of \hspace{.5pt} $\Gamma_M$, and that $\eta_M$ factors through $f$. Then the functor $\Hom_{\C}(M,-) \otimes_{\Gamma_M} \Gamma_M /(f)$ induces a triangle functor
\begin{center}
\begin{tikzpicture}
\diagram{d}{3em}{4em}{
\left ( \HFact_d \left ( \C,S, \eta \right ), \Sigma, \Delta \right ) & \K_d \Gamma_M/(f) \\
 };
\path[->, font = \scriptsize, auto]
(d-1-1) edge (d-1-2);
\end{tikzpicture}
\end{center}
It maps a $d$-fold factorization $(M_i,f_i) \in \HFact_d ( \C,S, \eta )$ to the $d$-complex
\begin{center}
\begin{tikzpicture}
\diagram{d}{3em}{3em}{
\cdots & \frac{\Hom_{\C} \left ( M,M_1 \right )}{\Hom_{\C} \left ( M,M_1 \right ) \cdot f} & \cdots &  \frac{\Hom_{\C} \left ( M,M_d \right )}{\Hom_{\C} \left ( M,M_d \right ) \cdot f} & \cdots \\
 };
\path[->, font = \scriptsize, auto]
(d-1-1) edge node{$\overline{(S^{-1}f_d)_*}$} (d-1-2)
(d-1-2) edge node{$\overline{(f_1)_*}$} (d-1-3)
(d-1-3) edge node{$\overline{(f_{d-1})_*}$} (d-1-4)
(d-1-4) edge node{$\overline{(f_d)_*}$} (d-1-5);
\end{tikzpicture}
\end{center}
of right $\Gamma_M/(f)$-modules, and morphisms as in \emph{(1)}.
\end{theorem}

\begin{proof}
For (1), note that a homotopy between morphisms in $\Fact_d ( \C,S, \eta )$ induces a homotopy between chain maps of sequences of right $\Gamma_M$-modules, in a natural way. Thus we obtain an additive functor $\HFact_d ( \C,S, \eta ) \to \HSeq \Gamma_M$. It commutes with the suspension in both categories, since this is just given by left shift of sequences. Finally, the functor preserves distinguished triangles, since these are defined in terms of mapping cones and standard triangles in both triangulated categories. 

For (2), note first that the tensor product $- \otimes_{\Gamma_M} \Gamma_M /(f)$ induces a triangle functor
\begin{center}
\begin{tikzpicture}
\diagram{d}{3em}{4em}{
\HSeq \Gamma_M & \HSeq \Gamma_M/(f) \\
 };
\path[->, font = \scriptsize, auto]
(d-1-1) edge (d-1-2);
\end{tikzpicture}
\end{center}
Composing this with the functor from (1) therefore gives a triangle functor
\begin{center}
\begin{tikzpicture}
\diagram{d}{3em}{4em}{
\left ( \HFact_d \left ( \C,S, \eta \right ), \Sigma, \Delta \right ) & \HSeq \Gamma_M/(f) \\
 };
\path[->, font = \scriptsize, auto]
(d-1-1) edge (d-1-2);
\end{tikzpicture}
\end{center}
as in the statement. We must show that the image of this functor consists of $d$-complexes. By assumpotion, the morphism $\eta_M$ factors through $f$, so there is a morphism $h \in \Hom_{\C} \left ( M, SM \right )$ with $\eta_M = h \circ f$.

Let us prove that the composition $\overline{(f_d)_*} \circ \overline{(f_{d-1})_*} \circ \cdots \circ \overline{(f_1)_*}$ is zero; the vanishing of the other $d$-fold compositions is proved in the same way. For an element $g \in \Hom_{\C} \left ( M,M_1 \right )$, the composition maps the image $\overline{g}$ in $\Hom_{\C} \left ( M,M_1 \right ) / \Hom_{\C} \left ( M,M_1 \right ) \cdot f$ to the element $\overline{f_d \circ \cdots \circ f_1 \circ g}$ in $\Hom_{\C} \left ( M,SM_1 \right ) / \Hom_{\C} \left ( M,SM_1 \right ) \cdot f$. Now note that $f_d \circ \cdots \circ f_1 = \eta_{M_1}$, and that $\eta$ being a natural transformation gives a commutative diagram
\begin{center}
\begin{tikzpicture}
\diagram{d}{3em}{4em}{
M & M_1 \\
SM & SM_1 \\
 };
\path[->, font = \scriptsize, auto]
(d-1-1) edge node{$g$} (d-1-2)
(d-2-1) edge node{$Sg$} (d-2-2)
(d-1-1) edge node{$\eta_M$} (d-2-1)
(d-1-2) edge node{$\eta_{M_1}$} (d-2-2);
\end{tikzpicture}
\end{center}
in $\C$. This gives
$$f_d \circ \cdots \circ f_1  \circ g = \eta_{M_1} \circ g = Sg \circ \eta_M = Sg \circ h \circ f$$
which is an element of the right $\Gamma_M$-module $\Hom_{\C} \left ( M,SM_1 \right ) \cdot f$. This shows that the composition $\overline{(f_d)_*} \circ \overline{(f_{d-1})_*} \circ \cdots \circ \overline{(f_1)_*}$ is zero.
\end{proof}

We shall now turn our attention to categories $\C$ of the form $\add M$ for an object $M \in \C$, where $\add M$ denotes the additive closure of $M$. Recall that this is the additive category whose objects are all the possible retracts of the finite direct sums of copies of $M$, i.e.\ every object $X \in \C$ for which there exist a positive integer $n$ and morphisms
\begin{center}
\begin{tikzpicture}
\diagram{d}{3em}{4em}{
X & M^{\oplus n} & X \\
 };
\path[->, font = \scriptsize, auto]
(d-1-1) edge node{$i$} (d-1-2)
(d-1-2) edge node{$p$} (d-1-3);
\end{tikzpicture}
\end{center}
with $p \circ i = 1_X$. When $\C$ is of this form for an object $M \in \C$, then $\Hom_{\C}(M,N)$ is a finitely generated  projective right $\Gamma_M$-module for every $N \in \C$, where $\Gamma_M = \Hom_{\C}(M,M)$. The theorem therefore takes the following form.

\begin{corollary}\label{cor:functors}
With the same assumptions as in \emph{Theorem \ref{thm:functors}}, suppose in addition that $\C = \add M$. Furthermore, denote by $\Proj ( \Gamma_M )$ and $\Proj ( \Gamma_M/(f) )$ the categories of finitely generated projective right $\Gamma_M$-modules and $\Gamma_M/(f)$-modules, respectively. Then the functor $\Hom_{\C}(M,-)$ induces a triangle functor
\begin{center}
\begin{tikzpicture}
\diagram{d}{3em}{4em}{
\left ( \HFact_d \left ( \C,S, \eta \right ), \Sigma, \Delta \right ) & \HSeq \Proj \left ( \Gamma_M \right ) \\
 };
\path[->, font = \scriptsize, auto]
(d-1-1) edge (d-1-2);
\end{tikzpicture}
\end{center}
and the functor $\Hom_{\C}(M,-) \otimes_{\Gamma_M} \Gamma_M /(f)$ induces a triangle functor
\begin{center}
\begin{tikzpicture}
\diagram{d}{3em}{4em}{
\left ( \HFact_d \left ( \C,S, \eta \right ), \Sigma, \Delta \right ) & \K_d \Proj \left ( \Gamma_M/(f) \right ) \\
 };
\path[->, font = \scriptsize, auto]
(d-1-1) edge (d-1-2);
\end{tikzpicture}
\end{center}
\end{corollary}

Before we specialize further, we include two elementary lemmas. The first one is well known in the case of modules over rings, but we were unable to find a reference for the general case.

\begin{lemma}\label{lem:elementary1}
Let $\C$ be an additive category, $M \in \C$ and object, and denote $\Hom_{\C}(M,M)$ by $\Gamma_M$. Then the functor $\Hom_{\C}(M,-)$, from $\C$ to the category of right $\Gamma_M$-modules, is fully faithful on $\add M$.
\end{lemma}

\begin{proof}
Given objects $X,Y \in \C$, the group homomorphism
\begin{center}
\begin{tikzpicture}
\diagram{d}{3em}{4em}{
\Hom_{\C}(X,Y) & \Hom_{\Gamma_M} \left ( \Hom_{\C}(M,X), \Hom_{\C}(M,Y) \right ) \\
 };
\path[->, font = \scriptsize, auto]
(d-1-1) edge (d-1-2);
\end{tikzpicture}
\end{center}
maps a morphism $f$ to $f_*$, with $f_*(g) = f \circ g$. It is bijective when $X=M$: in this case the right hand side is just $\Hom_{\Gamma_M} ( \Gamma_M, \Hom_{\C}(M,Y) )$, and so when we compose with the natural evaluation isomorphism
\begin{center}
\begin{tikzpicture}
\diagram{d}{3em}{4em}{
\Hom_{\Gamma_M} \left ( \Gamma_M, \Hom_{\C}(M,Y) \right ) & \Hom_{\C}(M,Y) \\
 };
\path[->, font = \scriptsize, auto]
(d-1-1) edge (d-1-2);
\end{tikzpicture}
\end{center}
the result is the identity on $\Hom_{\C}(M,Y)$. By direct sum arguments, the homomorphism is also bijective when $X = M^{\oplus n}$ for $n \ge 1$, and finally also when $X$ is a retract of such an object. 
\end{proof}

The second lemma is also most likely known, but again we were unable to find a reference.

\begin{lemma}\label{lem:elementary2}
Let $\Gamma$ be a ring and $x \in \Gamma$ a central element. Then for every finitely generated projective right $\Gamma$-module $P$, there is an isomorphism
$$\Hom_{\Gamma} \left ( P, \Gamma \right ) / \Hom_{\Gamma} \left ( P, \Gamma \right ) x \simeq \Hom_{\Gamma / (x)} \left ( P/Px, \Gamma / (x) \right )$$
of abelian groups, natural in $P$.
\end{lemma}

\begin{proof}
Note first that since $x$ is a central element of $\Gamma$, we obtain for each $f \in \Hom_{\Gamma} \left ( P, \Gamma \right )$ a well defined homomorphism $fx \in \Hom_{\Gamma} \left ( P, \Gamma \right )$ given by $(fx)(p) = f(px)$ for $p \in P$. Thus $\Hom_{\Gamma}( P, \Gamma )x$ is a subgroup of $\Hom_{\Gamma}( P, \Gamma )$.

Consider the natural group homomorphism
\begin{center}
\begin{tikzpicture}
\diagram{d}{3em}{4em}{
\Hom_{\Gamma} \left ( P,A \right ) & \Hom_{\Gamma / (x)} \left ( P/Px, \Gamma / (x) \right ) \\
 };
\path[->, font = \scriptsize, auto]
(d-1-1) edge node{$\tau$} (d-1-2);
\end{tikzpicture}
\end{center}
given by reduction modulo $x$. Since $\Hom_{\Gamma}( P, \Gamma )x$ is contained in the kernel of $\tau$, we obtain an induced homomorphism
\begin{center}
\begin{tikzpicture}
\diagram{d}{3em}{4em}{
\Hom_{\Gamma} \left ( P,A \right ) / \Hom_{\Gamma} \left ( P, \Gamma \right ) x & \Hom_{\Gamma / (x)} \left ( P/Px, \Gamma / (x) \right ) \\
 };
\path[->, font = \scriptsize, auto]
(d-1-1) edge node{$\overline{\tau}$} (d-1-2);
\end{tikzpicture}
\end{center}
One checks that this is an isomorphism when $P = \Gamma$, hence also when $P$ is finitely generated free, and finally also when $P$ is a direct summand of a finitely generated free module. Finally, one checks that the isomorphism is natural in $P$.
\end{proof}

We now restrict the homotopy category $\HFact_d \left ( \C,S, \eta \right )$ to the case when $d=2$ and $S = 1_{\C}$. By definition, $\eta$ is then a natural transformation $1_{\C} \to 1_{\C}$, in particular $\eta_M$ is a central element in $\Hom_{\C}(M,M)$ for every object $M \in \C$. The objects of $\HFact_2 \left ( \C,1_{\C}, \eta \right )$ are diagrams
\begin{center}
\begin{tikzpicture}
\diagram{d}{3em}{4em}{
X & Y & X \\
 };
\path[->, font = \scriptsize, auto]
(d-1-1) edge node{$f$} (d-1-2)
(d-1-2) edge node{$g$} (d-1-3);
\end{tikzpicture}
\end{center}
in $\C$, with $g \circ f = \eta_X$ and $f \circ g = \eta_Y$. 

The following result shows that when $\C = \add M$ for some $M \in \C$, and $\eta_M$ is a regular element of $\Hom_{\C}(M,M)$, then the second functor in Corollary \ref{cor:functors} is actually fully faithful. Moreover, its image consists of \emph{totally acyclic} complexes of finitely generated projective right $\Gamma_M / ( \eta_M )$-modules. Recall that a complex $C$ (of, say, right modules) over a ring $\Gamma$ is totally acyclic if both $C$ and $\Hom_{\Gamma}(C, \Gamma )$ are acyclic, that is, exact. We denote by $\Ktac \Proj ( \Gamma )$ the homotopy category of such totally acyclic complexes of finitely generated  projective right modules; the triangulated structure is the same as for the homotopy category of complexes. Thus the following result shows that under the above assumptions, the category $( \HFact_2 ( \C, 1_{\C}, \eta ), \Sigma, \Delta )$ embeds as a triangulated subcategory into the category of totally acyclic complexes of finitely generated projective right $\Gamma_M/( \eta_M )$-modules, where $\Gamma_M = \Hom_{\C}(M,M)$.

\begin{theorem}\label{thm:fullyfaithful}
Let $\C$ be an additive category, and $\eta \colon 1_{\C} \to 1_{\C}$ a natural transformation. Furthermore, suppose that $\C = \add M$ for some object $M \in \C$, and that $\eta_M$ is a regular element of $\Hom_{\C}(M,M)$. Then the functor $\Hom_{\C}(M,-) \otimes_{\Gamma_M} \Gamma_M /( \eta_M )$ induces a fully faithful triangle functor
\begin{center}
\begin{tikzpicture}
\diagram{d}{3em}{4em}{
\left ( \HFact_2 \left ( \C, 1_{\C}, \eta \right ), \Sigma, \Delta \right ) & \Ktac \Proj \left ( \Gamma_M/( \eta_M ) \right ) \\
 };
\path[->, font = \scriptsize, auto]
(d-1-1) edge (d-1-2);
\end{tikzpicture}
\end{center}
where $\Gamma_M = \Hom_{\C}(M,M)$, and $\Proj ( \Gamma_M/( \eta_M) )$ denotes the category of finitely generated projective right $\Gamma_M/( \eta_M )$-modules.
\end{theorem}

\begin{proof}
For simplicity, let us denote the ring $\Gamma_M$ by just $\Gamma$. Moreover, given an object $X \in \C$, let us denote the projective right $\Gamma$-module $\Hom_{\C}(M,X)$ by $e(X)$. Finally, we denote the quotient ring $\Gamma /( \eta_M)$ by $\overline{\Gamma}$, and the projective right $\overline{\Gamma}$-module $e(X)/e(X) \cdot \eta_M$ by $\overline{e}(X)$.

The image of a factorization
\begin{center}
\begin{tikzpicture}
\diagram{d}{3em}{4em}{
X & Y & X \\
 };
\path[->, font = \scriptsize, auto]
(d-1-1) edge node{$f$} (d-1-2)
(d-1-2) edge node{$g$} (d-1-3);
\end{tikzpicture}
\end{center}
is the sequence
\begin{equation*}\label{sequence}\tag{$\dagger$}
\begin{tikzpicture}
\diagram{d}{3em}{4em}{
\cdots & \overline{e}(X) & \overline{e}(Y) & \overline{e}(X) & \overline{e}(Y) & \cdots \\
 };
\path[->, font = \scriptsize, auto]
(d-1-1) edge (d-1-2)
(d-1-2) edge node{$\overline{f_*}$} (d-1-3)
(d-1-3) edge node{$\overline{g_*}$} (d-1-4)
(d-1-4) edge node{$\overline{f_*}$} (d-1-5)
(d-1-5) edge (d-1-6);
\end{tikzpicture}
\end{equation*}
of finitely generated projective right $\overline{\Gamma}$-modules. By Corollary \ref{cor:functors} the sequence is a complex; we show first that it is totally acyclic.

Let $\overline{w}$ be an element in $\overline{e}(Y)$, represented by a morphism $w \in e(Y) = \Hom_{\C}(M,Y)$, and suppose that $0 = \overline{g_*} ( \overline{w} ) = \overline{ g \circ w }$. Then $g \circ w \in e(X) \cdot \eta_M$, and so $g \circ w = u \circ \eta_M$ for some $u \in \Hom_{\C}(M,X)$. This gives 
$$w \circ \eta_M = \eta_Y \circ w = f \circ g \circ w = f \circ u \circ \eta_M$$
where the first equality follows from the fact that $\eta$ is a natural transformation. Now since $\eta_M$ is a regular element of $\Gamma$, it is also regular on every projective right $\Gamma$-module, in particular on $e(Y)$. Therefore $w = f \circ u$, giving $\overline{w} = \overline{f \circ u} = \overline{f_*} ( \overline{u} )$. This shows that $\Ker \overline{g_*} = \Im \overline{f_*}$, and similarly $\Ker \overline{f_*} = \Im \overline{g_*}$. Consequently, the complex (\ref{sequence}) is acyclic.

To show that the complex is totally acyclic, we must show that it is acyclic after we have applied $\Hom_{\overline{\Gamma}}(-, \overline{\Gamma} )$ to it. However, by Lemma \ref{lem:elementary2}, this is equivalent to showing that the complex
\begin{center}
\begin{tikzpicture}
\diagram{d}{3em}{3em}{
\cdots & \frac{ \Hom_{\Gamma} \left ( e(X), \Gamma \right ) }{  \Hom_{\Gamma} \left ( e(X), \Gamma \right ) \cdot \eta_M } & \frac{ \Hom_{\Gamma} \left ( e(Y), \Gamma \right ) }{  \Hom_{\Gamma} \left ( e(Y), \Gamma \right ) \cdot \eta_M } & \frac{ \Hom_{\Gamma} \left ( e(X), \Gamma \right ) }{  \Hom_{\Gamma} \left ( e(X), \Gamma \right ) \cdot \eta_M }& \cdots \\
 };
\path[->, font = \scriptsize, auto]
(d-1-1) edge (d-1-2)
(d-1-2) edge node{$\overline{(g_*)^*}$} (d-1-3)
(d-1-3) edge node{$\overline{(f_*)^*}$} (d-1-4)
(d-1-4) edge (d-1-5);
\end{tikzpicture}
\end{center}
is acyclic. For the maps displayed in this complex, the lower star denotes the functor $\Hom_{\C}(M,-)$, while the upper star denotes the (contravariant) functor $\Hom_{\Gamma}(-, \Gamma )$. Now take an element $u \in \Hom_{\Gamma} ( e(Y), \Gamma )$, and suppose that $\overline{(f_*)^*}$ maps $\overline{u}$ to zero (where $\overline{u}$ denotes the element represented by $u$ in the quotient group displayed in the middle above). Then $u \circ f_*$ is an element of $\Hom_{\Gamma} ( e(X), \Gamma ) \cdot \eta_M$, and we can write $u \circ f_* = v \cdot \eta_M$ for some $v \in \Hom_{\Gamma} ( e(X), \Gamma )$. This gives 
$$u \circ \left ( \eta_Y \right )_* = u \circ \left ( f \circ g \right )_* = u \circ f_* \circ g_* = \left ( v \cdot \eta_M \right ) \circ g_*$$
as elements of $\Hom_{\Gamma} ( e(Y), \Gamma )$. Now take an element $w \in e(Y) = \Hom_{\C}(M,Y)$, and apply $u \circ ( \eta_Y )_*$ to it:
$$u \circ \left ( \eta_Y \right )_* \left ( w \right ) = u \left ( \eta_Y \circ w \right ) = u \left ( w \circ \eta_M \right ) = u \left ( w \right ) \circ \eta_M$$
Here, the second equality follows from the fact that $\eta$ is a linear transformation, and the third follows from the fact that $u$ is a homomorphism of right $\Gamma$-modules. Next, apply $( v \cdot \eta_M ) \circ g_*$ to $w$ instead:
$$\left ( v \cdot \eta_M \right ) \circ g_* \left ( w \right ) = \left ( v \cdot \eta_M \right ) \left ( g \circ w \right ) = v \left ( g \circ w \circ \eta_M \right ) = v \left ( g \circ w \right ) \circ \eta_M$$
Here, the second equality follows from the way $v \cdot \eta_M$ is defined as an element of $ \Hom_{\Gamma} ( e(X), \Gamma )$, and the third equality from the fact that $v$, like $u$ above, is right $\Gamma$-linear. Now since $u \circ ( \eta_Y )_* = ( v \cdot \eta_M ) \circ g_*$, and $\eta_M$ is regular on $\Gamma$, we see that $u(w) = v(g \circ w)$ for every element $w \in e(Y)$. This implies that $u = v \circ g_*$ as elements of $\Hom_{\Gamma} ( e(Y), \Gamma )$, giving
$$\overline{u} = \overline{v \circ g_*} = \overline{\left ( g_* \right )^*} \left ( \overline{v} \right )$$
We have now proved that $\Ker \overline{(f_*)^*} = \Im \overline{(g_*)^*}$, and similarly $\Ker \overline{(g_*)^*} = \Im \overline{(f_*)^*}$. Consequently, the the complex (\ref{sequence}) is totally acyclic.

We have now proved that the image of a factorization in $( \HFact_2 ( \C, 1_{\C}, \eta ), \Sigma, \Delta )$ is a totally acyclic complex of projective right $\overline{\Gamma}$-modules. Since the triangulated structure of $\Ktac \Proj ( \overline{\Gamma} )$ is the same as in the homotopy category $\K \Proj ( \overline{\Gamma} )$, we see from Corollary \ref{cor:functors} that we now have a functor
\begin{center}
\begin{tikzpicture}
\diagram{d}{3em}{4em}{
\left ( \HFact_2 \left ( \C, 1_{\C}, \eta \right ), \Sigma, \Delta \right ) & \Ktac \Proj ( \overline{\Gamma} ) \\
 };
\path[->, font = \scriptsize, auto]
(d-1-1) edge (d-1-2);
\end{tikzpicture}
\end{center}
of triangulated categories. Let us call this functor $F$; we must now show that it is both full and faithful.

To show that $F$ is full, we start with two factorizations
\begin{center}
\begin{tikzpicture}
\diagram{d}{2em}{4em}{
X & Y & X \\
U & V & U \\
 };
\path[->, font = \scriptsize, auto]
(d-1-1) edge node{$f$} (d-1-2)
(d-1-2) edge node{$g$} (d-1-3)
(d-2-1) edge node{$p$} (d-2-2)
(d-2-2) edge node{$q$} (d-2-3);
\end{tikzpicture}
\end{center}
in $( \HFact_2 ( \C, 1_{\C}, \eta ), \Sigma, \Delta )$, and a morphism $\varphi$ in $\Ktac \Proj ( \overline{\Gamma} )$ represented by a chain map
\begin{center}
\begin{tikzpicture}
\diagram{d}{3em}{4em}{
\cdots & \overline{e}(X) & \overline{e}(Y) & \overline{e}(X) & \overline{e}(Y) & \cdots \\
\cdots & \overline{e}(U) & \overline{e}(V) & \overline{e}(U) & \overline{e}(V) & \cdots \\
 };
\path[->, font = \scriptsize, auto]
% upper row
(d-1-1) edge (d-1-2)
(d-1-2) edge node{$\overline{f_*}$} (d-1-3)
(d-1-3) edge node{$\overline{g_*}$} (d-1-4)
(d-1-4) edge node{$\overline{f_*}$} (d-1-5)
(d-1-5) edge (d-1-6)
% lower row
(d-2-1) edge (d-2-2)
(d-2-2) edge node{$\overline{p_*}$} (d-2-3)
(d-2-3) edge node{$\overline{q_*}$} (d-2-4)
(d-2-4) edge node{$\overline{p_*}$} (d-2-5)
(d-2-5) edge (d-2-6)
% verticals
(d-1-2) edge node{$\varphi_2$} (d-2-2)
(d-1-3) edge node{$\varphi_1$} (d-2-3)
(d-1-4) edge node{$\varphi_0$} (d-2-4)
(d-1-5) edge node{$\varphi_{-1}$} (d-2-5);
\end{tikzpicture}
\end{center}
Now adapt the proof of \cite[Proposition 3.4]{BerghJorgensen}; this is possible since $\eta_M$ is central and regular in $\Gamma$. We then obtain a commutative diagram
\begin{equation*}\label{diagram}\tag{$\dagger \dagger$}
\begin{tikzpicture}
\diagram{d}{3em}{4em}{
e(X) & e(Y) & e(X) \\
e(U) & e(V) & e(U) \\
 };
\path[->, font = \scriptsize, auto]
% horizontals
(d-1-1) edge node{$f_*$} (d-1-2)
(d-1-2) edge node{$g_*$} (d-1-3)
(d-2-1) edge node{$p_*$} (d-2-2)
(d-2-2) edge node{$q_*$} (d-2-3)
% verticals
(d-1-1) edge node{$a$} (d-2-1)
(d-1-2) edge node{$b$} (d-2-2)
(d-1-3) edge node{$a$} (d-2-3);
\end{tikzpicture}
\end{equation*}
of projective right $\Gamma$-modules, in such a way that when we reduce modulo $\eta_M$ there are two diagonal homomorphisms
\begin{center}
\begin{tikzpicture}
\diagram{d}{3em}{4em}{
\cdots & \overline{e}(X) & \overline{e}(Y) & \overline{e}(X) & \overline{e}(Y) & \cdots \\
\cdots & \overline{e}(U) & \overline{e}(V) & \overline{e}(U) & \overline{e}(V) & \cdots \\
 };
\path[->, font = \scriptsize, auto]
% upper row
(d-1-1) edge (d-1-2)
(d-1-2) edge node{$\overline{f_*}$} (d-1-3)
(d-1-3) edge node{$\overline{g_*}$} (d-1-4)
(d-1-4) edge node{$\overline{f_*}$} (d-1-5)
(d-1-5) edge (d-1-6)
% lower row
(d-2-1) edge (d-2-2)
(d-2-2) edge node{$\overline{p_*}$} (d-2-3)
(d-2-3) edge node{$\overline{q_*}$} (d-2-4)
(d-2-4) edge node{$\overline{p_*}$} (d-2-5)
(d-2-5) edge (d-2-6)
% verticals
(d-1-2) edge node{$\overline{a} - \varphi_2$} (d-2-2)
(d-1-3) edge node{$\overline{b} - \varphi_1$} (d-2-3)
(d-1-4) edge node{$\overline{a} - \varphi_0$} (d-2-4)
(d-1-5) edge node{$\overline{b} - \varphi_{-1}$} (d-2-5)
% diagonals
(d-1-5) edge node[above]{$s_{-1}$} (d-2-4)
(d-1-4) edge node[above]{$s_0$} (d-2-3);
\end{tikzpicture}
\end{center}
satisfying $\overline{a} - \varphi_0 = s_{-1} \circ \overline{f_*} + \overline{q_*} \circ s_0$. As explained in the last part of the proof of \cite[Proposition 3.4]{BerghJorgensen}, these diagonal homomorphisms can be extended in both directions to a nullhomotopy; this uses only the fact that we are dealing with totally acyclic complexes of finitely generated projective (right) $\overline{\Gamma}$-modules. Therefore the periodic chain map $( \dots, \overline{a}, \overline{b}, \overline{a}, \dots )$ also represents the morphism $\varphi$ in $\Ktac \Proj ( \overline{\Gamma} )$.

Now consider the commutative diagram (\ref{diagram}) of projective right $\Gamma$-modules, and recall from Lemma \ref{lem:elementary1} that the functor $e(-)$ is fully faithful on $\C$. Since it is full, we can write the vertical homomorphisms as $a = \alpha_*$ and $b = \beta_*$, where $\alpha \colon X \to U$ and $\beta \colon Y \to V$ are morphisms in $\C$. Moreover, since the functor is faithful, the diagram
\begin{center}
\begin{tikzpicture}
\diagram{d}{3em}{4em}{
X & Y & X \\
U & V & U \\
 };
\path[->, font = \scriptsize, auto]
% horizontals
(d-1-1) edge node{$f$} (d-1-2)
(d-1-2) edge node{$g$} (d-1-3)
(d-2-1) edge node{$p$} (d-2-2)
(d-2-2) edge node{$q$} (d-2-3)
% verticals
(d-1-1) edge node{$\alpha$} (d-2-1)
(d-1-2) edge node{$\beta$} (d-2-2)
(d-1-3) edge node{$\alpha$} (d-2-3);
\end{tikzpicture}
\end{center}
commutes in $\C$, and therefore represents a morphism $\theta$ in $( \HFact_2 ( \C, 1_{\C}, \eta ), \Sigma, \Delta )$. From the above we see that $F( \theta ) = \varphi$, and this proves that the functor $F$ is full.

To show that  $F$ is faithful, we start with a morphism $\theta$ as above, and suppose that $F( \theta ) = 0$ in $\Ktac \Proj ( \overline{\Gamma} )$. In other words, there exists a nullhomotopy
\begin{center}
\begin{tikzpicture}
\diagram{d}{3em}{4em}{
\cdots & \overline{e}(X) & \overline{e}(Y) & \overline{e}(X) & \overline{e}(Y) & \cdots \\
\cdots & \overline{e}(U) & \overline{e}(V) & \overline{e}(U) & \overline{e}(V) & \cdots \\
 };
\path[->, font = \scriptsize, auto]
% upper row
(d-1-1) edge (d-1-2)
(d-1-2) edge node{$\overline{f_*}$} (d-1-3)
(d-1-3) edge node{$\overline{g_*}$} (d-1-4)
(d-1-4) edge node{$\overline{f_*}$} (d-1-5)
(d-1-5) edge (d-1-6)
% lower row
(d-2-1) edge (d-2-2)
(d-2-2) edge node{$\overline{p_*}$} (d-2-3)
(d-2-3) edge node{$\overline{q_*}$} (d-2-4)
(d-2-4) edge node{$\overline{p_*}$} (d-2-5)
(d-2-5) edge (d-2-6)
% verticals
(d-1-2) edge node{$\overline{\alpha_*}$} (d-2-2)
(d-1-3) edge node{$\overline{\beta_*}$} (d-2-3)
(d-1-4) edge node{$\overline{\alpha_*}$} (d-2-4)
(d-1-5) edge node{$\overline{\beta_*}$} (d-2-5)
% diagonals
(d-1-3) edge node[above]{$s_1$} (d-2-2)
(d-1-4) edge node[above]{$s_0$} (d-2-3)
(d-1-5) edge node[above]{$s_{-1}$} (d-2-4);
\end{tikzpicture}
\end{center}
This time we adapt the proof of \cite[Proposition 3.3]{BerghJorgensen}; we can do this for the same reason as before. As a result, we obtain two diagonal homomorphisms
\begin{center}
\begin{tikzpicture}
\diagram{d}{3em}{4em}{
e(X) & e(Y) & e(X) \\
e(U) & e(V) & e(U) \\
 };
\path[->, font = \scriptsize, auto]
% horizontals
(d-1-1) edge node{$f_*$} (d-1-2)
(d-1-2) edge node{$g_*$} (d-1-3)
(d-2-1) edge node{$p_*$} (d-2-2)
(d-2-2) edge node{$q_*$} (d-2-3)
% verticals
(d-1-1) edge node{$\alpha_*$} (d-2-1)
(d-1-2) edge node{$\beta_*$} (d-2-2)
(d-1-3) edge node{$\alpha_*$} (d-2-3)
% diagonals
(d-1-2) edge node[above]{$y$} (d-2-1)
(d-1-3) edge node[above]{$x$} (d-2-2);
\end{tikzpicture}
\end{center}
of right $\Gamma$-modules, with $\alpha_* = y \circ f_* + q_* \circ x$ and $\beta_* = x \circ g_* + p_* \circ y$. As above, since the functor $e(-)$ is fully faithful by Lemma \ref{lem:elementary1}, we can write $y = s_*$ and $x = t_*$ for some morphisms $s \colon Y \to U$ and $t \colon X \to V$ in $\C$, with $\alpha = s \circ f + q \circ t$ and $\beta = t \circ g + p \circ s$. Therefore, the diagram
\begin{center}
\begin{tikzpicture}
\diagram{d}{3em}{4em}{
X & Y & X \\
U & V & U \\
 };
\path[->, font = \scriptsize, auto]
% horizontals
(d-1-1) edge node{$f$} (d-1-2)
(d-1-2) edge node{$g$} (d-1-3)
(d-2-1) edge node{$p$} (d-2-2)
(d-2-2) edge node{$q$} (d-2-3)
% verticals
(d-1-1) edge node{$\alpha$} (d-2-1)
(d-1-2) edge node{$\beta$} (d-2-2)
(d-1-3) edge node{$\alpha$} (d-2-3)
% diagonals
(d-1-2) edge node[above]{$s$} (d-2-1)
(d-1-3) edge node[above]{$t$} (d-2-2);
\end{tikzpicture}
\end{center}
displays a nullhomotopy for the morphism representing $\theta$, showing that $\theta = 0$ in $( \HFact_2 ( \C, 1_{\C}, \eta ), \Sigma, \Delta )$. This proves that the functor $F$ is faithful.
\end{proof}

Here is an example.

\begin{example}\label{ex:generalmodule}
Let $R$ be a commutative ring, $M$ an $R$-module, and $x \in R$ fixed element. Then multiplication by $x$ induces a natural transformation $\eta_x \colon 1_{\Mod R} \to 1_{\Mod R}$, where $\Mod R$ is the category of $R$-modules. The transformation restricts to a natural transformation $\eta_x \colon 1_{\add M} \to 1_{\add M}$, and we can form the homotopy category $( \HFact_2 ( \add M, 1_{\add_M}, \eta_x ), \Sigma, \Delta )$. If $x$ is $M$-regular, then $(\eta_x)_M$, which is just multiplication by $x$, is a regular element of $\Hom_R(M,M)$. Then from Theorem \ref{thm:fullyfaithful} we obtain a fully faithful triangle functor
\begin{center}
\begin{tikzpicture}
\diagram{d}{3em}{4em}{
\left ( \HFact_2 \left ( \add M, 1_{\add_M}, \eta_x \right ), \Sigma, \Delta \right ) & \Ktac \Proj \left ( \Gamma_M/(x) \right ) \\
 };
\path[->, font = \scriptsize, auto]
(d-1-1) edge (d-1-2);
\end{tikzpicture}
\end{center}
where $\Gamma_M = \Hom_R(M,M)$. For example, when $M = R$, then $\add M = \proj R$, and we recover the classical fully faithful triangle functor 
\begin{center}
\begin{tikzpicture}
\diagram{d}{3em}{4em}{
\left ( \HFact_2 \left ( \Proj (R), 1_{\Proj (R)}, \eta_x \right ), \Sigma, \Delta \right ) & \Ktac \Proj \left ( R/(x) \right ) \\
 };
\path[->, font = \scriptsize, auto]
(d-1-1) edge (d-1-2);
\end{tikzpicture}
\end{center}
from the homotopy category of matrix factorizations of $x$ to the homotopy category of totally acyclic complexes of finitely generated projective modules over $R/(x)$; see Example \ref{ex:matrixfactorizations}. This is an equivalence when $R$ is a regular local ring.
\end{example}

The previous example treats the trivial case when $R$ is a commutative ring and $M = R$; in this case, $\Hom_R(M,M)$ is just $R$ itself, so that Theorem \ref{thm:fullyfaithful} is just a reformulation of \cite[Theorem 3.5]{BerghJorgensen}. As mentioned, when $R$ is a regular local ring, then the functor is an equivalence, by \cite{Buchweitz} and \cite{Orlov}. However, the following result shows that the functor in Theorem \ref{thm:fullyfaithful} is always an equivalence whenever $\Hom_{\C}(M,M)$ is a commutative regular local ring, and not just in the trivial case above.

\begin{theorem}\label{thm:regularlocal}
With the same assumptions as in \emph{Theorem \ref{thm:fullyfaithful}}, suppose that $\Gamma_M$ is a commutative regular local ring, and that $\eta_M$ is an element in the square of its maximal ideal. Then the triangle functor
\begin{center}
\begin{tikzpicture}
\diagram{d}{3em}{4em}{
\left ( \HFact_2 \left ( \C, 1_{\C}, \eta \right ), \Sigma, \Delta \right ) & \Ktac \Proj \left ( \Gamma_M/( \eta_M ) \right ) \\
 };
\path[->, font = \scriptsize, auto]
(d-1-1) edge (d-1-2);
\end{tikzpicture}
\end{center}
is an equivalence.
\end{theorem}

\begin{proof}
By Theorem \ref{thm:fullyfaithful}, we only need to show that the functor is dense. Since $\Gamma_M$ is a commutative regular local ring, reduction modulo $\eta_M$ induces a triangle equivalence from the homotopy category of matrix factorizations of $\eta_M$ (over $\Gamma_M$) to $\Ktac \Proj ( \Gamma_M/( \eta_M ) )$; this is essentially proved in \cite{Eisenbud}. Since our functor 
\begin{center}
\begin{tikzpicture}
\diagram{d}{3em}{4em}{
\left ( \HFact_2 \left ( \C, 1_{\C}, \eta \right ), \Sigma, \Delta \right ) & \Ktac \Proj \left ( \Gamma_M/( \eta_M ) \right ) \\
 };
\path[->, font = \scriptsize, auto]
(d-1-1) edge (d-1-2);
\end{tikzpicture}
\end{center}
is $\Hom_{\C}(M,-)$ followed by reduction modulo $\eta_M$, we are done if we can show that every matrix factorization of $\eta_M$ over $\Gamma_M$ is in the image of the functor $\Hom_{\C}(M,-)$. However, this is a direct consequence of Lemma \ref{lem:elementary1}, together with the fact that since $\Gamma_M$ is local, every projective module is free and therefore of the form $\Hom_{\C}(M,M^n)$ for some integer $n \ge 1$.
\end{proof}

We end with an example.

\begin{example}\label{ex:regularlocal}
Let $k$ be a field, $k \llbracket x,y \rrbracket$ the power series ring in two variables, and $R$ the hypersurface $k \llbracket x,y \rrbracket / (xy)$. Furthermore, let $I$ be the ideal if $R$ generated by $x$, and put $\C = \add I$. A homomorphism $f \in \Hom_R(I,I)$ is uniquely determined by its action on $x$: since $xy = 0$ in $R$, there is a power series $p_f(x) \in k \llbracket x \rrbracket$ with $f(x) = p_f(x)x$. The assignment $f \mapsto p_f$ is easily seen to be an isomorphism $\Hom_R(I,I) \to k \llbracket x \rrbracket$ of rings, hence $\Hom_R(I,I)$ is a commutative regular local ring. For any $n \ge 2$, the element $x^n$ is regular on the ideal $I$, hence from Theorem \ref{thm:regularlocal} we obtain a triangle equivalence
\begin{center}
\begin{tikzpicture}
\diagram{d}{3em}{4em}{
\left ( \HFact_2 \left ( \C, 1_{\C}, x^n \right ), \Sigma, \Delta \right ) & \Ktac \Proj \left ( \Gamma_I/( x^n ) \right ) \\
 };
\path[->, font = \scriptsize, auto]
(d-1-1) edge (d-1-2);
\end{tikzpicture}
\end{center}
We may of course replace $\Ktac \Proj ( \Gamma_I/( x^n ) )$ with $\Ktac \Proj ( k \llbracket x \rrbracket /( x^n ) )$.

Next, let $k \llbracket x,y,z \rrbracket$ be the power series ring in three variables, and $S$ the quotient ring $k \llbracket x,y,z \rrbracket / (xz,yz)$. Let $J$ be the ideal of $S$ generated by $x$, and put $\C = \add J$. Given a homomorphism $g \in \Hom_S(J,J)$, there is a unique power series $q_g(x,y) \in k \llbracket x,y \rrbracket$ with $g(x) = q_g(x,y)x$, and the assignment $g \mapsto q_g$ is an isomorphism $\Hom_S(J,J) \to k \llbracket x,y \rrbracket$ of rings. Therefore $\Hom_S(J,J)$ is commutative regular local, and so since $xy$ is regular on $J$ we obtain a triangle equivalence
\begin{center}
\begin{tikzpicture}
\diagram{d}{3em}{4em}{
\left ( \HFact_2 \left ( \C, 1_{\C}, xy \right ), \Sigma, \Delta \right ) & \Ktac \Proj \left ( \Gamma_J/( xy ) \right ) \\
 };
\path[->, font = \scriptsize, auto]
(d-1-1) edge (d-1-2);
\end{tikzpicture}
\end{center}
where we may replace $\Ktac \Proj ( \Gamma_J/( xy ) )$ with $\Ktac \Proj ( k \llbracket x,y \rrbracket /( xy ) )$.
\end{example}

\end{document}